  \documentclass[]{article}
\usepackage[utf8]{inputenc}

\usepackage{pdfpages}

\usepackage{euler} 

\usepackage[Conny]{fncychap}

\usepackage[T1]{fontenc}
\usepackage[english]{babel}
\usepackage{mathrsfs}
\usepackage{amsmath} 
\usepackage{amsthm}  

\usepackage{amssymb,amsfonts,amscd}

\usepackage{comment}

\usepackage{xfrac}
\usepackage{faktor}
\usepackage{mathtools}
\usepackage{tikz}
\usetikzlibrary{matrix,arrows,decorations.pathmorphing,decorations.markings, cd, scopes, backgrounds}

\makeatletter
\tikzset{
	open/.code     = {\tikzset{right hook->, circled};},
	closed/.code   = {\tikzset{right hook->, slashed};},
	open'/.code    = {\tikzset{left hook->, circled};},
	closed'/.code  = {\tikzset{left hook->, slashed};},
	circled/.code  = {\tikzset{markwith = {\draw (0,0) circle (.375ex);}};},
	slashed/.code  = {\tikzset{markwith = {\draw[-] (-.4ex,-.4ex) -- (.4ex,.4ex);}};},
	markwith/.code ={
		\pgfutil@ifundefined%
		{tikz@library@decorations.markings@loaded}%
		{\pgfutil@packageerror{tikz}{You need to say %
				\string\usetikzlibrary{decorations.markings} to use arrows with markings}{}}{}%
		\pgfkeysalso{/tikz/postaction = {
				/tikz/decorate,
				/tikz/decoration={markings, mark = at position 0.5 with {#1}}}
		}
	},
}
\makeatother

\usepackage{centernot}
\usepackage{bbding}
\usepackage{indentfirst}
\usepackage{stackrel}
\usepackage{scalerel}
\usepackage{hyperref}
\usepackage{cleveref}
\usepackage{mathdots}
\usepackage{caption}
\usepackage{geometry}
\usepackage{etoolbox}
\usepackage{calligra}
\usepackage{leftidx}
\usepackage{tcolorbox}
\usepackage{bbm}
\usepackage{bm}
\usepackage[autostyle]{csquotes}
\usepackage{epigraph}

\usepackage{diagbox}

\usepackage{graphicx}

\usepackage{stmaryrd}

\usepackage{mathrsfs}
\DeclareMathAlphabet{\mathpzc}{OT1}{pzc}{m}{it}

\usepackage[toc,page]{appendix}

\usepackage[style=alphabetic, backend=bibtex, useprefix, hyperref]{biblatex}

\NewBibliographyString{appendix,by}
\DefineBibliographyStrings{english}{
	appendix = {with an appendix},
	by       = {by},
}

\newbibmacro*{appendixby}{%
	\iffieldundef{related}
	{}
	{\setunit{\addspace}%
		\bibstring{appendix}%
		\setunit{\addspace}%
		\bibstring{by}%
		\setunit{\addspace}%
		\entrydata{\strfield{related}}{\printnames{author}}}}

\AtEveryBibitem{%
	\ifentrytype{article}{%
		\usebibmacro{appendixby}%
	}{}%
}

\hypersetup{colorlinks=false, pdfborder={0 0 0}}
\usepackage[]{enumitem}
\setdescription{font=\normalfont}       
\makeatletter                                                  
\def\cleardoublepage{\clearpage\if@twoside \ifodd\c@page\else  
	\hbox{}                                                        
	\vspace*{\fill}                                                
	\begin{center}                                                 
		\*                                                             
	\end{center}                                                   
	\vspace{\fill}                                                 
	\thispagestyle{empty}                                          
	\newpage                                                       
	\if@twocolumn\hbox{}\newpage\fi\fi\fi}                         
\makeatother                                                   

\makeatletter
\renewcommand\part{%
	\if@openright
	\cleardoublepage
	\else
	\clearpage
	\fi
	\thispagestyle{empty}
	\if@twocolumn
	\onecolumn
	\@tempswatrue
	\else
	\@tempswafalse
	\fi
	\null\vfil
	\secdef\@part\@spart}
\makeatother

\usepackage{fancyhdr}                                   

\usepackage{xcolor}
\usepackage{url}
\usepackage{attachfile} 
\makeatletter                        
\g@addto@macro{\UrlBreaks}{\UrlOrds} 
\makeatother                         
\newsavebox\MBox
\makeatletter                        
\newcommand{\dotminus}{\mathbin{\text{\@dotminus}}}

\newcommand{\@dotminus}{%
	\ooalign{\hidewidth\raise1ex\hbox{.}\hidewidth\cr$\m@th-$\cr}%
}
\makeatother

\newcommand{\into}{\hookrightarrow}

\newcommand{\allora}{\Rightarrow}

\newcommand{\sseq}{\subseteq}

\newcommand{\set}[1]{ \left \{ #1 \right \} }

\newcommand{\mr}{\mathrm}
\newcommand{\mc}[1]{\mathcal{#1}}
\newcommand{\mb}{\mathbb}
\newcommand{\mbbm}{\mathbbm}
\newcommand{\mf}{\mathfrak}

\newcommand{\Z}{\mathbb{Z}}

\newcommand{\pro}{\mathbb{P}}
\newcommand{\A}{\mathbb{A}}
\newcommand{\sk}{\mathbbm{k}}
\newcommand{\st}[2]{\mathrel{\raisebox{#1 pt}{$ \mathrel{\stretchto{\mid}{#2 ex}} $}}}

\newcommand{\mo}[1]{\mc O_{#1}}

\newcommand{\catname}[1]{\mathbf{#1}}

\newcommand{\oocatname}[1]{\scaleobj{1.25}{\mathpzc{#1}}}

\newcommand{\colim}[1]{\underset{#1}{\mathrm{colim}}\ }

\newcommand{\epfs}{{}_{\#}}

\newcommand\mydiagbox[2]{\hbox{\tabcolsep=\arraycolsep\diagbox{$#1$}{$#2$}}}

\newcommand{\Th}[2]{\mr{Th}_{#1}\left( #2 \right)}

\newcommand{\restrict}[2]{{
		\left.\kern-\nulldelimiterspace 
		#1 
		\vphantom{\big|} 
		\right|_{#2} 
}}

\newcommand{\bigslant}[2]{
	\mathchoice
	{
		{\raisebox{0em}{$#1$}\!\!\!\;\!\!\;\left/\!\!\;\!\!\;\raisebox{-0em}{$#2$}\right.}%
	}
	{
		#1\:\!/\:\!#2
	}
	{
		#1\:\!/\:\!#2
	}
	{
		#1\:\!/\:\!#2
	}
}

\newcommand{\quot}[2]{
	\mathchoice
	{
		{\raisebox{.2em}{$#1$}\!\!\,\left/\!\raisebox{-.2em}{$#2$}\right.}%
	}
	{
		#1\!\:/\!\!\:\:#2
	}
	{
		#1\!\:/\!\!\:\:#2
	}
	{
		#1\!\:/\!\!\:\:#2
	}
}

\makeatletter
\DeclareRobustCommand*{\mfaktor}[3][]
{
	{ \mathpalette{\mfaktor@impl@}{{#1}{#3}{#2}} }
}
\newcommand*{\mfaktor@impl@}[2]{\mfaktor@impl#1#2}
\newcommand*{\mfaktor@impl}[4]{
	\settoheight{\faktor@zaehlerhoehe}{\ensuremath{#1#2{#3}}}%
	\settoheight{\faktor@nennerhoehe}{\ensuremath{#1#2{#4}}}%
	\raisebox{-0.5\faktor@zaehlerhoehe}{\ensuremath{#1#2{#3}}}%
	\mkern-4mu\reflectbox{\,$ / $\,}\mkern-5mu%
	\raisebox{0.5\faktor@nennerhoehe}{\ensuremath{#1#2{#4}}}%
}
\makeatother

\makeatletter
\newcommand{\bigperp}{%
	\mathop{\mathpalette\bigp@rp\relax}%
	\displaylimits
}

\newcommand{\bigp@rp}[2]{%
	\vcenter{
		\m@th\hbox{\scalebox{\ifx#1\displaystyle2.1\else1.5\fi}{$#1\perp$}}
	}%
}
\makeatother

\newcommand\blfootnote[1]{%
	\begingroup
	\renewcommand\thefootnote{}\footnote{#1}%
	\addtocounter{footnote}{-1}%
	\endgroup
}

\DeclareMathOperator{\Hom}{Hom}

\DeclareMathOperator{\spec}{Spec}

\DeclareMathOperator{\sym}{Sym}

\DeclareMathOperator{\Map}{Map}
\DeclareMathOperator{\iMap}{\underline{Map}}

\makeatletter
\g@addto@macro\bfseries{\boldmath}
\makeatother

\makeatletter
\newcommand{\cbigotimes}{\DOTSB\cbigotimes@\slimits@}
\newcommand{\cbigotimes@}{\mathop{\widehat{\bigotimes}}}
\makeatother

\newtheorem{thm}{Theorem}
\numberwithin{thm}{section} 
\newtheorem{co}[thm]{Corollary}

\newtheorem{lemma}[thm]{Lemma}

\newtheorem{pr}[thm]{Proposition}

\theoremstyle{definition}
\newtheorem{defn}[thm]{Definition}

\newtheorem{rmk}[thm]{Remark}
\newtheorem{notation}[thm]{Notation}

\newtheorem{construction}[thm]{Construction}
\newtheorem{exa}[thm]{Example}

\newtheorem{ithm}{Theorem}
\newtheorem{ipr}[ithm]{Proposition}

\fancypagestyle{main}{\fancyhf{}
	\fancyhead[LE,RO]{\scriptsize $\mr{KW}$-Euler Classes}
	\fancyhead[RE,LO]{\scriptsize \rightmark }
	\fancyfoot[CE,CO]{\thepage}}

\title{$\mr{KW}$-Euler Classes\\
	\large via Twisted Symplectic Bundles}

\author{Alessandro D'Angelo }
\date{}
 
\bibliography{KW-Euler_Biblio.bib}
\begin{document}

		\maketitle{}
	\begin{abstract}
		In this paper we are going to compute the $ \mr{KW} $-Euler classes for rank 2 vector bundles on the classifying stack $ \mc {B}N $, where $N$ is the normaliser of the standard torus in $SL_2$ and $\mr{KW}$ represents  Balmer's derived Witt groups. Using these computations we will recover, through a new and different strategy, the formulas previously obtained by Levine in Witt-sheaf cohomology. In order to obtain our results, we will prove K\"unneth formulas for products of $GL_n$'s and $SL_n$'s classifying spaces and we will develop from scratch the basic theory of twisted symplectic bundles with their associated twisted Borel classes in $SL$-oriented theories. 
	\end{abstract}
	
	\tableofcontents
	
	\section*{Introduction}

	 One of the main reasons to extend intersection theoretic techniques to motivic homotopy theory is that, one can get much richer invariants. Indeed, in the motivic homotopy category $ \mr{SH}(\sk) $, by a celebrated theorem of Morel, we have that  the endomorphisms of the sphere spectrum corresponds to quadratic forms in the Grothendieck-Witt ring: $ \mr{End}_{\mr{SH}(\sk)}(\mbbm 1_{\sk})\simeq \mr{GW}(\sk) $.  In the recent years, this led to much progress in the field, now called $ \A^1 $-enumerative geometry, with a lot of new interesting results by Hoyois, Kass, Levine, McKean, Pauli, Solomon, Wendt, Wickelgren and many others (for example one can look in \cite{Bachmann_Wickelgren} and references therein).\\

	 In this paper we are going to compute the $ \mr{KW} $-Euler classes for some special rank 2 vector bundles of the ind-scheme $ {B}N $, where $N$ is the normaliser of the standard torus in $SL_2$. 
	 The main source of inspiration was \cite{Motivic_Euler_Char}, where Levine computed Euler classes in Witt-sheaf cohomology. In our attempt to reproduce similar formulas for the spectrum $ \mr{KW} $, representing Balmer's derived Witt groups, we had to dive into some K\"unneth formulas and into the theory of twisted symplectic bundles obtaining different results that are of independent interest.
	 
	 
	 One can think of Witt sheaf cohomology, used in \textit{loc. cit.}, and the Witt spectrum as quadratic analogues of Chow groups and K-theory. Indeed, Witt sheaf cohomology and Witt theory capture the essentially quadratic information in Chow-Witt groups and Hermitian K-theory, which are quadratic refinements of Chow groups and K-theory respectively.\\
	 
	 Notably, Witt sheaf cohomology is a homotopy module and, therefore, it is simpler to work with from a homotopical perspective: it is \textit{bounded} in the sense of \cite[Definition 4.13]{Levine_Atiyah-Bott}. However, $ SL_{\eta} $-oriented spectra, like Witt theory $ \mr{KW} $,  are not bounded. Care is needed when working with general spectra on algebraic stacks or ind-schemes; the necessary technical tools are provided in \cite{Motivic_Vistoli}.\\
	 
	 Here is a detailed synopsis of the paper. In the first section we will provide the reader with a recollection of facts ranging from intersection theory on algebraic stacks, to orientations and Thom isomorphisms in the motivic context. Then, we will then show how to reduce computations of characteristic classes from general vector bundles to special linear ones. In the third section we will deduce some K\"unneth formulas for $SL_{\eta}$-oriented theories: in this kind of generality, these formulas are new to the best of author knowledge, and of independent interest. In particular we will prove:
	 
	 \begin{ipr}[Proposition {\ref{ch3:_Kunneth_Formula} }]
	 	Let $ S \in \catname{Sm}_{\bigslant{}{\sk}} $. Let $ \mr{A} $ be an $ SL_{\eta} $-oriented ring spectrum. Let $ \mc X=\prod_{i=1}^{s} \mc{B}GL_{n_i,S} \times \prod_{j=s+1}^{s+r} \mc{B}SL_{n_j,S}  $. Then we have:
	 	
	 	\[ \mr{A}^{\bullet}(X; L)\simeq \cbigotimes_{i=1}^{s} \mr{A}^{\bullet}\left( \mc{B}GL_{n_i}; L_i \right) \widehat\otimes_{A(\sk)} \cbigotimes_{j=s+1}^{s+r} \mr{A}^{\bullet}\left( \mc{B}SL_{n_j}; L_j \right) \]
	 	\noindent with $ L:=L_1 \boxtimes \ldots \boxtimes L_{s+r} $ and $ L_i\in \mr{Pic}(\mc{B}GL_{n_i}) $ for $ i=1,\ldots,s $ and $ L_j\in \mr{Pic}(\mc{B}SL_{n_j}) $ for $ j=s+1,\ldots,s+r $ and where $ \widehat{\otimes} $ denotes the completed tensor product.
	 \end{ipr}
	 
	  In the fourth section, we will introduce \textit{twisted symplectic bundles} to handle formal ternary laws (in the sense of \cite{Déglise_Fasel_Borel_Char}) that will be crucial for our last computation, \Cref{ch3:_recursive_formulas_proposition}.  We will prove a twisted quaternionic projective bundle formula and a twisted version of the Cartan sum formula for $SL$-oriented spectra (cf. \Cref{ch3:_Twisted_Proj_Bun_Thm} and \Cref{ch3:_Cartan_Sum_Formula_Twisted} respectively):
	  
	  \begin{ithm}[Theorem {\ref{ch3:_Twisted_Proj_Bun_Thm}}]
	  		Let $ \mb E $ be an $ SL $-oriented ring spectrum with a twisted Thom structure. Let $ (V, \omega^L) $ be a twisted symplectic bundle of rank 2n over a scheme $ X\in \catname{Sm}_{\bigslant{}{S}} $, let $ (\mc U, \restrict{\omega^L}{\mc U}) $ be the tautological rank 2 bundle over $ \mr H\pro^n_L $ and let $ \zeta:=b(\mc U, \restrict{\omega^L}{\mc U}) $ be its Borel class. Write $ \pi: \mr H\pro^n_L \rightarrow X $ for the projection map. Then for any closed subset $ Z \sseq X $ we have the isomorphism  of $ \mb E(X) $-modules:
	  	\[ (1,\zeta,\zeta^2,\ldots,\zeta^{n-1}): \bigoplus_{j=0}^{n-1}\mb E_Z(X; L^{\otimes -j}) \longrightarrow \mb E_{\pi^{-1}(Z)}(\mr H\pro^n_L) \]
	  \end{ithm}

	  We will then use the theory of twisted Borel classes to extend \cite[Lemma 8.2]{Ananyevskiy_FTL_Witt} to general vector bundles in \Cref{ch3:_Ana_Stab_Op_Lemma}. Finally, we will end providing the formulas for the Euler classes of the bundles $\widetilde{\mo{}}^{\pm}(m) $ (already introduced in \cite{Motivic_Euler_Char}), getting a complete description in terms of recursing formulas:
	  
	  \begin{ithm}[{\Cref{ch3:_recursive_formulas_proposition}}]
	  		In $\mr{KW}\left(  BN, \det(\widetilde{\mo{}}^{\pm}(m)) \right) $, we have the following:
	  		\begin{enumerate}
	  			\item For $ m=2n+1 $:
	  			\begin{equation} 
	  				b_1\left(\widetilde{\mo{}}^{+}(m)\right)=\sum_{k=0}^{n} (-1)^{n-k}\alpha_{k,n}\gamma^{k}e^{2k+1}
	  			\end{equation}
	  			\noindent with $ e=e(\widetilde{\mo{}}^{+}(1) ) $ and $ \alpha_{k,n} $ defined by recurrence relations.
	  			\item For $ m=2n $:
	  			\begin{equation} \label{ch3:_Even_recursive_formula}
	  				b_1\left(\widetilde{\mo{}}^{+}(m)\right)=\tilde{e}\left(\sum_{k=0}^{n-1} (-1)^{n-k+1}\beta_{k,n}\gamma^{k}e^{2k}\right)
	  			\end{equation}
	  			\noindent with $ e=e(\widetilde{\mo{}}^{+}(1) ) $, $\tilde e=e(\widetilde{\mo{}}^{+}(2) )$  and $ \beta_{k,n} $ defined by recurrence relations.
	  		\end{enumerate}
	  \end{ithm}
	 

	\section*{Acknowledgements}

This paper was part of the author's PhD thesis and he would like to thank his supervisor M. Levine for the countless discussions and insights he patiently shared with him: this work could have not been accomplished without his encouragement and inspiration. The author would also like to thank C. Chowdhury for many helpful discussions and the whole ESAGA group at the University of Essen for providing a stimulating working environment.

\blfootnote{The author was supported by the ERC through the project QUADAG.  This work is part of a project that has received funding from the European Research Council (ERC) under the European Union's Horizon 2020 research and innovation programme (grant agreement No. 832833).} \blfootnote{\hspace{-2em}\includegraphics[scale=0.08]{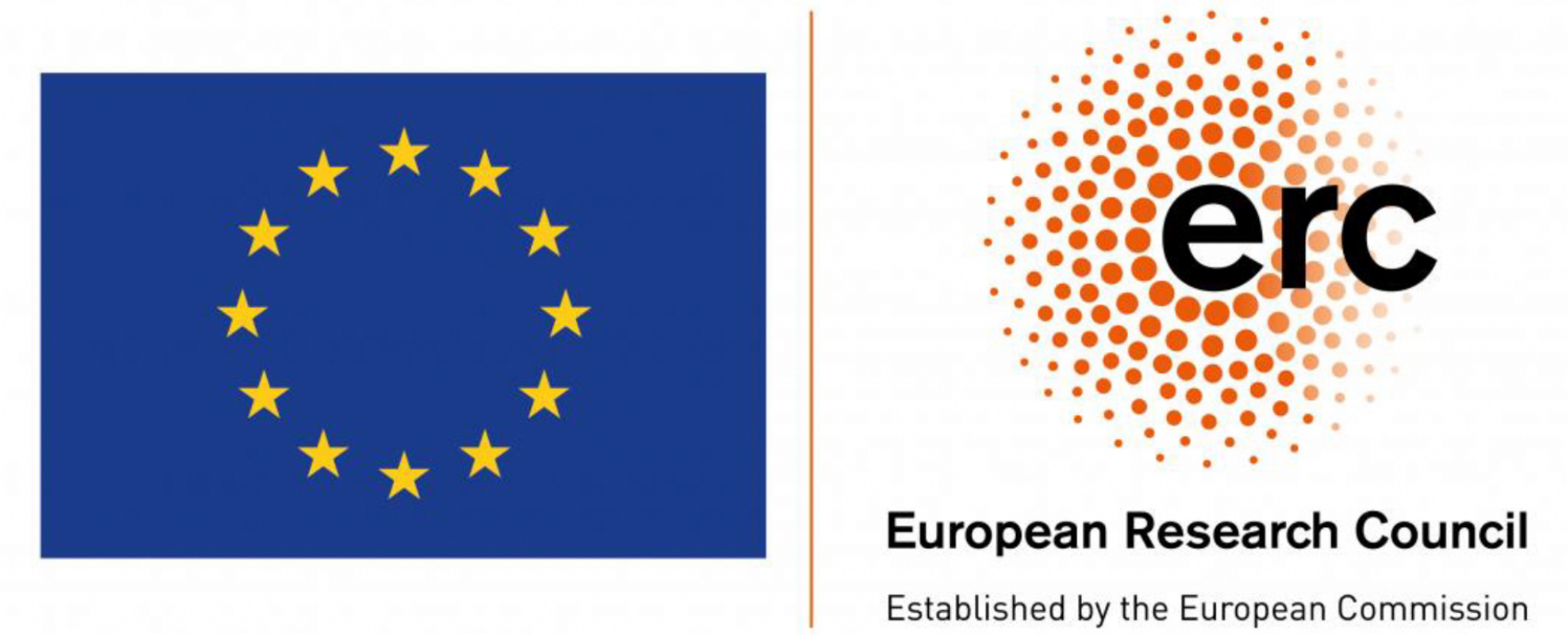}}.\\
	
	\section*{Notations and Conventions}
	
	\begin{enumerate}
		\item The categories $ \catname{Sch}_{\bigslant{}{B}}, \catname{Sch}_{\bigslant{}{B}}^{G} $ will always denote quasi-projective schemes over a base scheme $ B $ of finite Krull dimension, without or with a (left) $ G $-action. If we write $ \catname{Sm}_{\bigslant{}{B}},\catname{Sm}_{\bigslant{}{B}}^{G} $ then we only consider smooth quasi-projective schemes.
		\item We will denote by $ \oocatname{ASt}_{\bigslant{}{S}} $ the $ \infty $-category of algebraic stacks over some base $ S $ (that could be either a scheme or a stack, but for us will always be a scheme). Moreover given $ X \in \catname{Sch}_{\bigslant{}{B}}^{G} $, we will denote the associated quotient stack as $ \left[\bigslant{X}{G}\right] $.
		\item If not specified otherwise, whenever we are working over a field $ \sk $, we will assume it is of characteristic different from $ 2 $. If we work over a general base scheme $ B $, we will always assume $ \dfrac{1}{2} \in \mo{B}^{\times} $.
		\item Recall that a morphism of quasi-projective schemes $ f: X \rightarrow S $ in $ \catname{Sch}_{\bigslant{}{B}} $ is called \textit{lci} (that stands for \textit{local complete intersection}) if there exists a factorization of $ f $ as $ X \stackrel{i}{\into} M \stackrel{p}{\rightarrow} S $ with $ i $ a regular closed immersion and $ p $ a smooth map. In the conventions of \cite{DJK}, these are called \textit{smoothable lci} maps, but we do not need such distinction.
		
		\item If not specified otherwise, $ G $ will always denote a closed sub-group scheme inside $ GL_n $ for some $ n $.
		
		\item Given a group $ S $-scheme $ G $, we will always denote by $ \mf g^{\vee}_S $ the sheaf associated to the co-Lie algebra of $ G $. If the base scheme is clear from the context, we will only write $ \mf g^{\vee} $.
		
		\item In \cite[\S 4]{Morel-Voevodsky}, some ind-schemes were introduced to approximate quotient stacks. For a given algebraic group $ G $, those ind-schemes were denoted in \textit{loc. cit.} as \textit{geometric} classifying spaces $ B_{gm}G $. To distinguish between actual quotient classifying stacks $ [\bigslant{S}{G}]\in \oocatname{ASt}_{\bigslant{}{S}} $, over some base $ S $, and geometric classifying spaces (that are just ind-schemes), we will use a dual notation:
		\[ \mc BG:=\left[\bigslant{S}{G}\right] \in \oocatname{ASt}_{\bigslant{}{S}}\]
		\[  BG:= B_{gm}G \in \mr{Ind}(\catname{Sch}_{\bigslant{}{S}})\]
		We will try to use the calligraphic font $ 	\mc X, \mc Y, \mc BG $, etc., for algebraic stacks.

		\item Given a scheme (or an algebraic stack) $ X $, we will denote its Thomason-Trobaugh K-theory space as $ \mr K(X)=\mr K(\oocatname{Perf}(X)) $, where $ \oocatname{Perf}(X) $ is the infinity category of perfect complexes on $ X $. Given $ X \in \catname{Sch}_{\bigslant{}{B}}^{G} $ with associated quotient stack $ \mc X:= \left[ \bigslant{X}{G} \right] $, the \textit{genuine} equivariant $ \mr K $-theory of $ \mc X $ will be denoted as $ \mr K^G(X):=\mr K(\oocatname{Perf^G}(X))=\mr K(\mc X) $.
		\item Given infinity category $ \oocatname{C} $, we will denote by $ \Map_{\oocatname{C}}(-,-) $ the mapping space of $ \oocatname{C} $. If moreover $ \oocatname{C} $ is closed monoidal, the internal mapping space of $ \oocatname{C} $ we will denoted as $ \iMap_{\oocatname{C}}(-,-) $.
		
		\item When dealing with \textit{bivariant} theories, as known as Borel-Moore homology for us, we will have to take into account \textit{twists} by perfect complexes. For example given a scheme (or an algebraic stack) $ X $ and a perfect complex $ v \in \oocatname{Perf}(X) $, we can define the $ v $-twisted Borel-Moore homology $ \mb E(X,[v] ) $ as done in \cite{DJK}, where $ [v]\in K_0(X)=K_0(\oocatname{Perf}(X)) $. When $ \mc V $ is a locally free sheaf with associated class $ v:=[\mc V]\in K_0(X) $, to the automorphism in $ \mr{SH}(X) $ given by \textit{$ v $-suspension} $ \Sigma^{v} $, it will then correspond the element $ \Th{X}{V}:=\Sigma^{v}\mbbm 1_X $, where $ V:=\mb V_{X}(\mc V):=\spec(\sym^{\bullet}(\mc V^{\vee})) $ is the vector bundle given by $ \mc V $. We will always use calligraphic letters $ \mc V,\mc E $, etc., for perfect complexes and Roman letters for vector bundles.  We are also using calligraphic letters for algebraic stacks, but this should not be source of any confusion since it will be clear from the context to which kind of object we are referring to.

		\item We will stick with the conventions and notation of \cite{DJK} for the general theory of motivic bivariant theories. But from the second section and on, we will use the twisting conventions of \cite{Motivic_Euler_Char}, \cite{VLF_Levine}, \cite{Levine_Atiyah-Bott} that clashes with the one of \cite{DJK} when we  twist by line bundles: what in the former three papers  is denoted $ \mb E^{a,b}\left( X; L \right) $ is actually $ \mb E^{a+2,b+1}\left( \Th{X}{L} \right) $, that in the notation of the latter one is $ \mb E^{a+2,b+1}\left( X; -[\mc L] \right) $ where $ \mc L $ is the invertible sheaf corresponding to the line bundle $ L $. More precisely we have:
		
		\[  \mb E^{a,b}_{\mr{Levine}}\left( X; L \right):=\mb E^{a+2,b+1}\left( \Th{X}{L} \right) =: \mb E^{a+2,b+1}_{\mr{DJK}}\left( X,  -[\mc L] \right)  \]
		\noindent where the double index conventions on the right hand side are defined in \cref{ch1:_DJK_classical_recap}.
		Although slightly confusing, both notation have their pros and cons. It should be clear  what we are using from the context and from the fact that we will always use the semi-colon for one and just a comma for the other (and square brackets if we want to stress that we are considering K-theoretic classes). The rule of the thumb should be: if we are working with cohomology theories and if it is a twist by a line bundle (and not by an invertible sheaf), we are using the $ \mb E_{\mr{Levine}}(-;-) $-convention, otherwise we are following \cite{DJK}.
		
	\end{enumerate}

\section{Recollections}
\subsection{Geometric Approximations}
\label{sec.2:_Factorization_U_j}
We will now work over a base scheme $ S $ and we will denote by $ GL_n $ the group scheme $ GL_{n,S}:=GL_n\times_{\Z} S $ defined over $ S $, where $ GL_n $ is the usual group scheme of invertible $ (n\times n )$-matrices defined over $ \spec(\Z) $. Whenever we will encounter a group scheme $ G $ over $ S $, we will always assume that $ G $ is smooth.\\
We describe a version of the construction of the classifying space of $ G $, found in \cite[\S 4.2]{Morel-Voevodsky}. We can consider  $ V_m=\A^{n(n+m)}_{S}\simeq \Hom(\A^{n+m}_{S}, \A^{n}_{S}) $ equipped with the natural (left) $ GL_n $-action (hence we also have an induced natural $ G $-action).  Once we identify $ V_m $ with the scheme of $ (n,n+m) $-matrices, we can restrict to the open subset $ {E}_mG\sseq V_m $ made of those matrix with rank $ n $. On $ {E}_mG $ we have a free $ GL_n $-action (hence a free $ G $-action too) . We can  define a map $ s_m: {E}_mG \into E_{m+1}G $ sending an element $ B \in  {E}_mG $ to a matrix of the form:

\[ \left(\begin{array}{ccc|c}
	& & &0\\
	& B & &\vdots \\
	& & & 0 \\
\end{array}\right) \]

\begin{notation}\label{ch1:_Notation_quot_ind}
	Since we will closely follow  \cite{Levine_Atiyah-Bott} and \cite{VLF_Levine}, we will adapt the same notation too.  Choose as a base point   $ x_0=(I_n, 0_n, \ldots, 0_n) $ where $ I_n $ is the $ (n\times n) $ identity matrix and $ 0_n $ is the zero vector column of length $ n $.  We will denote by  $ {E}SL_n={E}GL_n $ the presheaf on $ \catname{Sm}_{S} $ given by $ \colim{m} \ {E}_mSL_n $. For any closed subgroup $ G $ of $ GL_n $ or $ SL_n $, we will denote the quotient $ {B}_mG:=\bigslant{{E}_mGL_n}{G} $ whose limit gives us the approximation $ {B}G:=\colim{m} {B}_mG $ for the quotient stack $ \mc BG:=\left[ \bigslant{S}{G} \right] $. For $ G=GL_n $, the spaces $ {B}_mGL_n\simeq \mr{Gr}_S(n, n+j) $ are represented by the Grassmannians. 
\end{notation}

\subsection*{The Normaliser of the Torus in $SL_2$}

Let $ N $ be the normaliser of the standard diagonal torus $ T\sseq SL_2 $. Note that $ T\simeq \mb G_m $, where for $ R $ a ring, we map $ t \in \mb G_m(R)= R^{\times} $ to the diagonal matrix:
\[ \left( \begin{matrix}
	t & 0 \\
	0 & t^{-1}
\end{matrix} \right) \]
We often simply write this matrix as $ t $, when there is no cause of confusion.
Notice that $ N $ is generated by $ T $ plus the element:
\[  \sigma:=\left(\begin{matrix*}[c]
	0 &1 \\
	-1 & 0
\end{matrix*}\right) \]

Recall from \cite{Witt_Loc_PhD} the following:

\begin{lemma}
	Let $ \sk $ be a field. The Picard group $ \mr{Pic}(\mc BN) $ of $ \mc BN \in \oocatname{ASt}_{\bigslant{}{\sk}} $ is generated by the line bundle $ \gamma_N \in \mr{Pic}^{SL_2}(\pro^2\setminus C)\simeq \mr{Pic}(\mc BN)  $ coming from $ \mo{\pro^2}(1) $, with its natural $ SL_2 $-linearisation. Moreover $ \mr{Pic}(\mc BN)\simeq \bigslant{\Z}{2\Z} $.
\end{lemma}

\begin{rmk}
	We will always identify the classifying space $\mc BN$ with its geometric approximation $ BN $ by \cite[Proposition 3.33]{Motivic_Vistoli}, so with some abuse of notation we will still denote by $\gamma$ the corresponding system of line bundles on $ BN $.
\end{rmk}

\begin{rmk}
	The representation $ \rho^{-}: N \rightarrow \mb G_m $ sending $ \sigma $ to $ -1 $ corresponds exactly to the line bundle $ \gamma_N$ generating $ Pic(\mc BN) $. 
\end{rmk}

\subsection{BM Motives and Operations on Algebraic Stacks}

We will freely use the six functor formalism developed for algebraic stacks in \cite{ChoDA24}. A summary is given by the following:

\begin{thm}[{\cite[Theorem 4.26]{ChoDA24}}]
	We have a functor:
	\[ \mr{SH}^*_!: Corr\left( \oocatname{ASt} \right)_{lft,all} \longrightarrow \oocatname{Pr}_{stb}^{\mr L} \]
	\noindent extending the analogous functor defined on schemes. This functor encodes the following data:
	
	\begin{enumerate}
		\item For every $ \mc X $ NL-stack, we have the tensor- hom adjunction in $ \mr{SH}(\mc X) $:
		\[ -\otimes - \dashv \iMap_{\mr{SH}(\mc X)}(-,-) \]
		\item For any map $ f: \mc X \longrightarrow \mc Y $ in $ \oocatname{ASt} $, we have a pair of adjoint functors:
		\[ f^*\dashv f^* \]
		\noindent and if $ f $ is smooth we also have:
		\[ f_{\#}\dashv f^* \dashv f_* \]
		\item For $ f: \mc X \longrightarrow \mc Y $ a lft map in $ \oocatname{ASt} $, we have another pair of adjoint functors:
		\[ f_!\dashv f^! \]
		\item Given a cartesian diagram in $ \oocatname{ASt}$: 
		\begin{center}
			\begin{tikzpicture}[baseline={(0,-1)}, scale=1.25]
				\node (a) at (0,1) {$ \mc W $};
				\node (b) at (1, 1) {$ \mc Y $};
				\node (c)  at (0,0) {$  \mc Z$};
				\node (d) at (1,0) {$ \mc X $};
				\node (e) at (0.2,0.75) {$ \ulcorner $};
				\node (f) at (0.5,0.5) {$ \Delta $};

				\path[font=\scriptsize,>= angle 90]
				
				(a) edge [->] node [above ] {$ g $} (b)
				(a) edge [->] node [left] {$ q $} (c)
				(b) edge[->] node [right] {$ p $} (d)
				(c) edge [->] node [below] {$ f $} (d);
			\end{tikzpicture}
		\end{center}
		\noindent where $p$ is lft, we have the base change equivalence:
		\[ p_!f^*\overset{Ex^*_!}{\simeq} g^*q_! \]
		\item For any lft map $f$, the projection formula holds:
		\[ f_!(-)\otimes - \simeq f_!( - \otimes f^*(-))  \]
		\item For any smooth map $f$, the smooth projection formula holds:
		\[ f\epfs(-)\otimes - \simeq f\epfs( - \otimes f^*(-))  \]
	\end{enumerate}
	These functors satisfy the usual compatibilities given by exchange transformations:
	$$ Ex^*_!,Ex^!_*, Ex^*_{\#}, Ex^*_*,Ex_{\# *}, Ex_{!*}, Ex_{!\#}, Ex^{*!} $$
	\noindent and they give rise to localisation fiber sequences like in \cite[Proposition 4.2.1]{Chowdhury24}. Moreover, for a given smooth map $f$, we have the following natural purity equivalence:
	\[ f\epfs \longrightarrow f_!\Sigma^{\mb L_f} \]
\end{thm}

\begin{defn}
	The adjoint equivalences:
	\[ \Sigma^{\mc E}:=p\epfs s_*: \mr{SH}(\mc X)\leftrightarrows \mr{SH}(\mc X): s^!p^*=:\Sigma^{-\mc E} \]
	\noindent are called \textit{Thom transformations}. We will denote by:
	\[ \Th{\mc X}{E}:=\Sigma^{\mc E}\mbbm 1_{\mc X} \in \mr{Pic}(\mr{SH}(\mc X)) \]
	\noindent the \textit{Thom space} of $ E $, with inverse $ \Sigma^{-\mc E}\mbbm 1_{\mc X} $.
\end{defn}

\begin{notation}
	Let:
	\[ J_{Bor}: \mr K \longrightarrow \mr{Pic}(\mr{SH}^{\triangleleft}) \]
	\noindent be the \textit{Borel J-homomorphism} as defined in \cite[\S 4.1]{ChoDA24}. For a given NL-stack $ \mc X $ and a given $ v \in \mr K_0(\mc X) $, we will denote the associated automorphism of $ \mr{SH}(\mc X) $ as $ \Sigma^{v} $, with inverse $ \Sigma^{-v} $. 
\end{notation}

\begin{defn}\label{ch1:_BM_Mot_and_thy_lci_ind_schemes}
	Let $ g: \mc X \rightarrow \mc B $ be a lft map of NL-stacks and $ w \in \mr K_0(\mc X) $ and let $ \mb E \in \mr{SH}(\mc B) $.
	
	\begin{enumerate}
		\item [{$\left( \begin{array}{c}
				\mr{\textit{BM}}\\
				\mr{\textit{Motive}}
			\end{array} \right)$}]  The twisted Borel-Moore motive over $ \mc B $ is defined as:
		\[ \left( \bigslant{\mc X}{\mc B} \right)^{\mr{BM}}(w):= g_! \Sigma^{w}  \mbbm 1_{\mc X} \]

		\item [{$\left( \begin{array}{c}
				\mr{\textit{BM}}\\
				\mr{\textit{Homology}}
			\end{array} \right)$}]  The twisted Borel-Moore homology over $ \mc B $ is defined as:
		\[ \mb E^{\mr{BM}}\left(\bigslant{\mc X}{\mc B}, w\right):=\iMap_{\mr{SH}(\mc B)}\left( \left( \bigslant{\mc X}{\mc B} \right)^{\mr{BM}}(w), \mb E \right) \]
		\noindent The BM-homology groups will be then defined as:
		\[  \mb E^{\mr{BM}}_{a,b}\left(\bigslant{\mc X}{\mc B}, w\right):=\pi_0\left( \Sigma^{-a,-b} \mb E^{\mr{BM}}\left(\bigslant{\mc X}{\mc B}, w\right) \right) \]
		
		\item [{$\left( \begin{array}{c}
				\mr{\textit{Generalised}}\\
				\mr{\textit{Cohomology}}
			\end{array} \right)$}]  The twisted  cohomology of $ \mc X $ is defined as:
		\[ \mb E\left( \mc X, w \right):=\iMap_{\mr{SH}(\mc B)}\left( \mbbm 1_{\mc B}, g_*\Sigma^{w}g^*\mb E  \right)\]
		\noindent and its twisted cohomology groups as:
		\[  \mb E^{a,b}\left( \mc X, w \right):=\pi_0\left( \Sigma^{a,b}  \mb E\left( \mc X, w \right) \right) \]
		
	\end{enumerate}
\end{defn}

As for the case of schemes (cf. \cite{DJK}), we can talk about smooth pushforwards, proper pullbacks and Gysin maps for maps of algebraic stacks, and associated operations in BM-homology and cohomology. 

\begin{defn}\label{ch1:_def_equiv_op}
	Let $ \pi_{\mc X}: \mc X \rightarrow \mc B $ and $ \pi_{\mc Y}: \mc Y \rightarrow \mc B $ be lft maps in  $ \oocatname{ASt}_{\bigslant{}{B}} $ and let $ f: \mc X \rightarrow \mc Y $ be a map between them. 
	
	\begin{enumerate}
		\item [{$\left( \begin{array}{c}
				\mr{\textit{SPf}}\\
			\end{array} \right)$}] If $ f $ is smooth and $ v \in \mr K_0(\mc Y) $, then we have a smooth pushforward map between BM-motives:
		\[ f_!: \left(  \bigslant{\mc X}{\mc B}\right)^{\mr{BM}}(v+ \mb L_f) \longrightarrow \left(  \bigslant{\mc Y}{\mc B}\right)^{\mr{BM}}(v)  \]
		\noindent induced by $ f_!\Sigma^{\mb L_f}f^*\stackrel{\mf p_f}{\simeq}  f_{\#}f^* \stackrel{\eta_{\#}^{*}(f)}{\longrightarrow} Id $.
		
		\item [{$\left( \begin{array}{c}
				\mr{\textit{PPb}}\\
			\end{array} \right)$}] If $ f $ is representable and proper, then we have a proper pullback map between BM-motives:
		\[ f^*:\left(  \bigslant{\mc Y}{\mc B}\right)^{\mr{BM}}(v) \longrightarrow \left(  \bigslant{\mc X}{\mc B}\right)^{\mr{BM}}(v)  \]

		\item [{$\left( \begin{array}{c}
				\mr{\textit{GPf}}\\
			\end{array} \right)$}] If $ f $ is smooth and it admits a section $ s: \mc Y \rightarrow \mc X $, then we have a natural transformation:
		\[ Id\simeq (f\circ s)_! (f\circ s)^!=f_!s_!s^!f^! \stackrel{\eta_!^!(s)}{\longrightarrow} f_! f^!\simeq f_!\Sigma^{\mb L_f} f^* \]
		This natural transformation induces then a \textit{Gysin pushforward} on BM-motives:
		\[ s_!: \left(  \bigslant{\mc Y}{\mc B}\right)^{\mr{BM}}(v) \longrightarrow \left(  \bigslant{\mc X}{\mc B}\right)^{\mr{BM}}(v+\mb L_f)   \]
		
		When $ f $ is a vector bundle, then by homotopy invariance the smooth pushforward $ f_! $ is an isomorphism on BM-motives with inverse given by the Gysin  pushforward $ s_! $ (the same argument in \cite[Lemma 2.2]{levine2017intrinsic} works verbatim).
	\end{enumerate}
	
	The operations we just defined on BM-motives will respectively induce smooth pullbacks, proper pushforwards and Gysin pullbacks on BM-homology as in the case of schemes. 
\end{defn}



\subsection{$ SL $- and $ SL_{\eta}$-Orientations }\label{ch2:_SL_orientations}
From now on in this chapter, we will assume our base field $ \sk $ to be perfect. We will denote by $ S $ our general base scheme. We have a very special element in $ \mr H^{-1,-1}(S) $, the algebraic Hopf map:
\[ \eta: \A^2_S\setminus \set{0} \longrightarrow \pro^1_S \]
\noindent sending $ (x,y)\mapsto [x:y] $, giving us an element $ \eta: \Sigma_{\mb G_m}\mbbm 1_{S} \rightarrow \mbbm 1_{S} \in \mr{H}^{-1,-1}(S)  $. For any motivic ring spectrum $ \mb E $, via the unit map $ \mu: \mbbm 1_{S} \rightarrow \mb E $, we get a corresponding element $ \eta_{\mb E} \in \mb E^{-1,-1}(S) $. 

\begin{defn}
A motivic ring spectrum $ \mb E \in \mr{SH}(S) $ is said to be $ \eta $-invertible if  multiplication by $ \eta_{\mb E} $, $ -\times \eta_{\mb E}: \mb E^{0,0}(S)\rightarrow \mb E^{-1,-1}(S) $, is an isomorphism.
\end{defn}

\begin{defn}[{\cite[Def.1]{Ananyevskiy_PhD_Thesis}}]
An $ SL_n $-vector bundle $ (E, \theta) $ on some $ X \in \catname{Sch}_{\bigslant{}{S}} $ is the data given by a vector bundle $ E $ of rank $ n $, together with a trivialization of the determinant, $ \theta: \det(E)\stackrel{\sim}{\rightarrow} \mo{X} $. We will often denote the $ SL_n $-vector bundle by just the underlying vector bundle $ E $. If there is no need to specify the rank of the bundle, we will say that $ E $ is a $ SL $-vector bundle.
\end{defn}

After Panin-Walter, we will use the following definition:
\begin{defn}[{cf. \cite{Ananyevskiy_SL_oriented}}]\label{ch2:_SL_orientation_def}
Let $ \catname{C} $ be a full subcategory of $ \catname{Sch}_{\bigslant{}{S}} $. Given a ring spectrum $ \mb E \in \mr{SH}(S) $, an \textit{SL-orientation with respect to} $ \catname{C} $ for $ \mb E $ is a rule which assigns to each $ SL_n $-vector bundle $ V $, over $ X \in \catname{C} $, an element:
\[ \mr{th}(V) \in \mb A^{2n,n}(\Th{X}{V}) \]
\noindent with the following properties:
\begin{enumerate}
	\item For any isomorphism $ \varphi: V_1\rightarrow V_2 $ of $ SL $-vector bundles over $ X \in \catname{C} $, we have:
	\[ \varphi^*\mr{th}(V_2)=\mr{th}(V_1) \]
	\noindent where $ \varphi^* $ is the pullback map induced by $ \varphi $.
	\item For any morphism $ f: X \rightarrow Y $ in $ \catname{C} $, and  $ V $ an $ SL $-vector bundle over $ Y $, we have:
	\[ f^*\mr{th}(V) =\mr{th}(f^*V)\]
	\item For $ \mc V_1, V_2 $  $ SL $-vector bundles on some $ X\in \catname{C} $, we have:
	\[ \mr{th}(V_1\oplus V_2)=p_1^*\mr{th}(V_1)\cup p_2^*\mr{th}(V_2) \]
	\noindent where $ p_i: V_1 \oplus V_2 \rightarrow V_i $ are the projection maps.
	\item We have:
	\[ \mr{th}(\A^1_S)=\Sigma_{T}1\simeq[ \Sigma_{T} \mbbm 1_S \stackrel{\Sigma_T u_{\mb E}}{\longrightarrow } \Sigma_{T} \mb E] \in \mb E^{2,1}(\pro^1_S) \]
	\noindent where $ u_{\mb E}: \mbbm 1_S \rightarrow \mb E $ is the unit map of the ring spectrum.
\end{enumerate}
We refer to the elements $ \mr{th}(V) $ as \textit{Thom classes}. If a ring spectrum $ \mb E $ has a normalised $ SL $-orientation with respect to $ \catname{C}:=\catname{Sm}_{\bigslant{}{S}} $, we simply say that $ \mb E $ has an $ SL $-orientation, and we will say that $ \mb E $ is $ SL $-oriented. If $ \catname{C}:=\catname{Sch}_{\bigslant{}{S}} $, then a normalised $ SL $-orientation with respect to $ \catname{C} $ will be called an \textit{absolute $ SL $-orientation}, and $ \mb E $ will be said to be absolutely $ SL $-oriented (following the conventions in \cite{Déglise_Fasel_Borel_Char}).

\end{defn}

\begin{rmk}
We will consider basically just absolute $ SL $-orientations. What will follow is already well known to the experts and it will be very similar to the material already presented in \cite[\S 4.3]{Bachmann_Wickelgren}. We only need to concentrate on $ SL $-oriented spectra and, by \cite[Theorem 4.7]{Ananyevskiy_Witt_MSp_Reations}, we know that those are strongly $ SL $-oriented in the sense of \cite{Bachmann_Wickelgren}, so the reader can safely refer to the latter if they prefer.
\end{rmk}

\begin{notation}\label{ch2:_SL_eta_oriented_def-notation}
Given an $ SL $-oriented spectrum $ \mr A \in \mr{SH}(S) $ that is also $ \eta $-invertible, we will say for short that $ \mr A $ is $ SL_{\eta} $-oriented. We will use the letter $ \mr A $ whenever we want to stress the fact that we are working with $ SL_{\eta} $-oriented spectra.
\end{notation}

\begin{rmk}
Following the conventions of \cite[Def.19]{Ananyevskiy_PhD_Thesis}, any $ \eta $-invertible spectrum $ \mr A $ will be regarded  just as a graded theory through the isomorphisms $ \mr A^{\bullet}:=\mr A^{a-b,0}\simeq  \mr A^{a,b}$ induced by $ \eta $. 
\end{rmk}

\begin{defn}\label{ch2:_Def_Thom_iso}
Let $ \catname{C} $ be a full subcategory of $ \catname{Sch}_{\bigslant{}{S}} $ and let $ \mb E \in \mr{SH}(S) $ be a motivic ring spectrum. A \textit{system of SL-Thom isomorphism} for $ \mb E $ (over $ \catname{C} $) is the data given by a collection of isomorphism $ \tau_V:  \mb E^{\bullet,\bullet}(X) \stackrel{\sim}{\rightarrow} \mb E^{\bullet+2n,\bullet+n}(\Th{X}{V}) $,  for $ X\in \catname{C} $ and $ V $ $ SL_n $-vector bundle on $ X $, such that:
\begin{enumerate}
	\item Given a map $ f: X \rightarrow Y $ in $ \catname{C} $ and $ V $ an $ SL_n $-vector bundle on $ X $, we have the following commutative diagram:
	\begin{center}
		\begin{tikzpicture}[baseline={(0,0)}, scale=2]
			\node (a) at (0,1) {$ \mb E^{\bullet,\bullet}(Y) $};
			\node (b) at (1.75, 1) {$ \mb E^{\bullet+2n,\bullet+n}\left( \Th{Y}{f^*V} \right) $};
			\node (c)  at (0,0) {$ \mb E^{\bullet,\bullet}(X)   $};
			\node (d) at (1.75,0) {$  \mb E^{\bullet+2n,\bullet+n}\left( \Th{X}{V} \right) $};
			\node (e) at (0.535,0.92) {$ \sim $};
			\node (f) at (0.54,0.05) {$ \sim $};

			\path[font=\scriptsize,>= angle 90]
			
			(a) edge [->] node [above ] {$ \tau_{f^*V} $} (b)
			(a) edge [->] node [left] {$  $} (c)
			(b) edge[->] node [right] {$  $} (d)
			(c) edge [->] node [below] {$ \tau_V $} (d);
		\end{tikzpicture}
	\end{center}
	\noindent where the vertical arrows are induced by the pullback on cohomology.
	\item Given an isomorphism $ \varphi: V \stackrel{\sim}{\rightarrow} W $ of $ SL_n $-vector bundles on $ X\in \catname{C} $, we get a commutative diagram: 
	\begin{center}
		\begin{tikzpicture}[baseline={(0,0)}, scale=1.5]
			\node (a) at (-1,0.3) {$ \mb E^{\bullet,\bullet}(X) $};
			\node (b) at (1.5, 0.6) {$ \mb E^{\bullet+2n,\bullet+n}\left( \Th{X}{V} \right) $};
			\node (d) at (1.5,0) {$  \mb E^{\bullet+2n,\bullet+n}\left( \Th{X}{W} \right) $};
			\node (e) at (1.35,0.3) {\scriptsize$ \varphi^* $};
			\node (f) at (0.5,0.5) {$ $};

			\path[font=\scriptsize,>= angle 90]
			
			(a) edge [->] node [above ] {$ \tau_{V} $} (b.west)
			(a) edge [->] node [below] {$ \tau_W $} (d.west)
			(b) edge[] node [right] {\hspace{-0.25em}\rotatebox{90}{$ \sim $}} (d);
		\end{tikzpicture}
	\end{center}
	\item  Given $ V_1, V_2 $ $ SL $-vector bundles of rank $ j,k $ over $ X \in \catname{C} $, the Thom isomorphism are \textit{multiplicative}, that is, we have the following commutative diagram:
	\begin{center}
		\begin{tikzpicture}[baseline={(0,0)}, scale=2]
			\node (a) at (0,1) {$ \mb E^{\bullet,\bullet}(X) \times \mb E^{\bullet,\bullet}(X)  $};
			\node (b) at (3.75, 1) {$ \mb E^{\bullet+2j,\bullet+j}\left( \Th{X}{V_1} \right)\times  \mb E^{\bullet+2k,\bullet+k}\left( \Th{X}{V_2} \right)  $};
			\node (c)  at (0,0) {$ \mb E^{\bullet,\bullet}(X)   $};
			\node (d) at (3.75,0) {$  \mb E^{\bullet+2(j+k),\bullet+j+k}\left( \Th{X}{V_1\oplus V_2} \right) $};
			\node (e) at (1.35,0.92) {$ \sim $};
			\node (f) at (1.35,0.05) {$ \sim $};

			\path[font=\scriptsize,>= angle 90]
			
			(a) edge [->] node [above ] {$ \tau_{V_1}\times \tau_{V_2} $} (b)
			(a) edge [->] node [left] {$  $} (c)
			(b) edge[->] node [right] {$  $} (d)
			(c) edge [->] node [below] {$ \tau_{V_1\oplus V_2} $} (d);
		\end{tikzpicture}
	\end{center}
	\noindent where the vertical arrows are induced by multiplication map of $ \mb E $, together with the identification $ \Th{X}{V_1}\otimes \Th{X}{V_2}\simeq \Th{X}{V_1\oplus V_2} $.

	\item If $ V $ is an $ SL_n $-vector bundle on $ X\in \catname{C} $ isomorphic to the trivial $ SL_n $-vector bundle $ \A^n_X $, then we have that $ \tau_{V}\simeq \Sigma^{2n,n} $.
	
\end{enumerate}
\end{defn}
\begin{rmk}\label{ch2:_rmk_Thom_iso}
If we had a collection of maps:
\[ \set{\tau_{V}:  \mb E^{\bullet,\bullet}(X) \stackrel{}{\rightarrow} \mb E^{\bullet+2n,\bullet+n}(\Th{X}{V}) }_{X \in \catname{C}}  \]
\noindent satisfying (1),(2), and (4) in the previous proposition, then we automatically get that $ \tau_V $ are isomorphism by a Mayer-Vietoris argument (cf. \cite[Lemma 3.7]{Ananyevskiy_Thom_iso}).
\end{rmk}

\begin{rmk}
Notice that working with a special group like $ SL $, giving Thom classes $ \mr{th}(V) $, for $ SL $-vector bundles $ V $ over $ X \in \catname{C} $, amounts to the same data as giving a system of Thom isomorphisms:
\[ \mb E^{\bullet,\bullet}(X) \stackrel{\sim}{\rightarrow} \mb E^{\bullet+2n,\bullet+n}(\Th{X}{V}) \]
From a system of Thom isomorphism, we can get a family of Thom classes just taking $ \mr{th}(V):=\tau_V(1) $. Vice versa, giving a family of Thom classes $ \set{\mr{th}(V)} $, we can define $ \tau_V(-):=\mr{th}(V)\cup p^*- $ with $ p: V \rightarrow X $ the projection map and $ -\cup-: \mb E^{\bullet,\bullet}(\Th{X}{X})\times \mb E^{\bullet,\bullet}(V)\rightarrow \mb E^{\bullet,\bullet}(\Th{X}{V}) $ the cup usual product map.\\
So an $ SL $-orientation will correspond to giving (a system of) Thom isomorphism for all $ X \in \catname{Sm}_{\bigslant{}{S}} $, while an absolute $ SL $-orientation will correspond to giving Thom isomorphisms for all $ X \in \catname{Sch}_{	\bigslant{}{S}} $.
\end{rmk}

\vspace{1cm}

By \cite[Example 16.30]{Bachmann_Hoyois_Norms_MHT} (applied to $ G=(SL_n)_n $), for any scheme $ X \in \catname{Sch}_{	\bigslant{}{S}}  $ and any $ V $ vector $ SL $-bundle of rank $ n $, we have an isomorphism $ \tau_V: \Sigma^V\mr{MSL}_X\stackrel{\sim}{\rightarrow} \Sigma^{2n,n} \mr{MSL}_X  $, where $ \mr{MSL}_X:=f^*MSL_S $ is the pullback of the special linear algebraic cobordism spectrum of \cite{Panin-Walter_MSL_MSp} along the structure map $ f: X \rightarrow S $. If we denote $ u_{\mr{MSL}_X}: \mbbm 1_X \rightarrow \mr{MSL}_X $ the unit map of $ \mr{MSL}_X $, then we have:
\[ \Sigma^{V}u_{\mr{MSL}_X}: \Sigma^V \mbbm 1_X \longrightarrow \Sigma^V \mr{MSL}_X  \stackrel{\tau_V}{\simeq} \Sigma^{2n,n}\mr{MSL}_X  \]
Notice that $ \Sigma^{V}u_{\mr{MSL}_X} $ lives in $ \mr{MSL}_X^{2n,n}(\Th{X}{V})\simeq\mr{MSL}_S^{2n,n}(X, -[V]) $ and it is not hard to check that these elements satisfy all the properties in \cref{ch2:_SL_orientation_def}.
\begin{defn}
For any $ X, V $ as above, we will denote the elements $ \mr{th}_{\mr{MSL}}(V):=\Sigma^{V}u_{\mr{MSL}_X}  \in \mr{MSL}_X^{2n,n}(\Th{X}{V})\simeq\mr{MSL}_S^{2n,n}(X, -[V])  $ and we will call $ \mr{th}_{\mr{MSL}}(V) $ the \textit{canonical $ \mr{MSL} $-Thom class} of $ V $.
\end{defn}

Canonical $ \mr{MSL} $-Thom classes give us an absolute orientation for $ \mr{MSL}_S $ that restricts on smooth schemes to the usual $ SL $-orientation of $ \mr{MSL}_S $. For any  $ SL $-oriented ring spectrum $\mb E\in \mr{SH}(S) $, by \cite[Theorem 4.7, Lemma 4.9]{Ananyevskiy_Witt_MSp_Reations}, there exists a ring spectrum map $ \varphi: \mr{MSL}_S \rightarrow \mb E $ such that $ \varphi(\mr{th}_{\mr{MSL}}(V))=\mr{th}_{\mb E}(V) $ for each smooth $ X \in \catname{Sm}_{\bigslant{}{S}} $ and any vector $ SL $-bundle $ V $ on $ X $.

\begin{rmk}\label{ch2:_BM_Thom_iso}
Once we have Thom classes and Thom isomorphism in cohomology for some $ \mb E \in \mr{SH} $, we will get Thom isomorphisms also in Borel-Moore homology using the cohomology product action on Borel-Moore homology (cf. \cite[\S 3.3]{Levine_Atiyah-Bott}). In particular if $ \mb E $ is $ SL $-oriented with respect to $ \catname{C} $, for any $ V $ vector $ SL $-bundle over $ X \in \catname{C}\sseq \catname{Sch}_{\bigslant{}{S}} $, of rank $ r $, we will have:
\[ \mb E^{\mr{BM}}_{a+2r,b+r}\left( \bigslant{X}{S} \right)\stackrel{\sim}{\longrightarrow }  \mb E^{\mr{BM}}_{a,b}\left( \bigslant{X}{S}, [V] \right) \]
\end{rmk}

\begin{defn}
Given $ \mb E\in \mr{SH}(S) $ an $ SL $-oriented ring spectrum and a map $ \varphi: \mr{MSL} \rightarrow \mb E$ of ring spectra, we call $ \varphi $ \textit{an $ SL $-orientation map}.
\end{defn}

Since $ \varphi $ is a map of ring spectra, if we define for any $ X \in \catname{Sch}_{\bigslant{}{S}} $ and any vector $ SL $-bundle $ V $ on $ X $:
\[ \mr{th}^{\varphi}_{\mb E}(V):=\varphi(\mr{th}_{\mr{MSL}}(V))  \in \mb E^{2n,n}(\Th{X}{V})  \]
\noindent we get an absolute $ SL $-orientation on $\mb E $ extending the given $ SL $-orientation we already had.

\begin{defn}
Consider $ \mb E\in \mr{SH}(S) $ an $ SL $-oriented ring spectrum, and suppose we are given an $ SL $-orientation map $ \varphi: \mr{MSL} \rightarrow \mb E $. Then we call the \textit{$ \varphi $-induced absolute $ SL $-orientation} the orientation data given by Thom classes:
\[ \mr{th}^{\varphi}_{\mb E}(V):=\varphi(\mr{th}_{\mr{MSL}}(V))  \in \mb E^{2n,n}(\Th{X}{V})  \]
\noindent for any $ X \in \catname{Sch}_{\bigslant{}{S}} $ and any vector $ SL $-bundle $ V $ on $ X $. For short we  will just say \textit{$ \varphi $-induced $ SL $-orientation}.
\end{defn}

We do not know a priori if the $ \varphi: \mr{MSL} \rightarrow \mb E $ that can be associated with an $ SL $-orientation is unique. While a similar unicity statement holds true for $ GL $- and $ Sp $-orientations by \cite[Remark 2.1.5]{Déglise_Fasel_Borel_Char}, for $ SL $-orientations is still open: there could be an obstruction preventing the uniqueness of $ \varphi $ living in $ \lim^1 \mb E^{2n-1,n}(\mr{MSL}_n^{fin}) $ by \cite[Theorem 5.8]{Panin-Walter_MSL_MSp}, where $ \mr{MSL}_n^{fin} $ are the Thom spaces associated to the tangent bundle of the special linear Grassmannian $ SGr(n,n^2) $ (cf. \cite[\S 5]{Panin-Walter_MSL_MSp},\cite[Def. 4.5]{Ananyevskiy_Witt_MSp_Reations}).

\begin{pr}\label{ch2:_Abs_VS_normal_Orientations}
Let $ \mr A \in \mr{SH}(\sk) $ be an $ \eta $-invertible motivic ring spectrum. Then $ SL $-orientations are in one to one correspondence with $ SL $-orientation maps $ \varphi: \mr{MSL}_{\sk}\longrightarrow \mr A $.
\end{pr}
\begin{proof}
Let $ SGr_{\sk}(n,m) $ the special linear Grassmannian, defined as the complement of the zero section of the determinant bundle associated to the universal bundle $ E(n,m) $ over the Grassmannian $ Gr_{\sk}(n,m) $ (see \cite{Ananyevskiy_PhD_Thesis} for more details). Let us denote $ \mr{Th}(n,m):=\Th{SGr(n,m)}{\oocatname{T}(n,m)} $ the Thom space associated to the tautological bundle $ \oocatname{T}(n,m) $ of the special linear Grassmannian $ SGr(n,m) $. Since $ \mr A $ is $ \eta $-invertible, we can adopt the single graded convention $ \mr A^{\bullet} $. By \cite[Theorem 5.8]{Panin-Walter_MSL_MSp}, we know that $ \varphi: \mr{MSL}_{\sk} \rightarrow \mr A $ as in our claim exists and the obstruction to the uniqueness of $ \varphi $ lies in $ \lim^1 \mr A^{n-1}(\mr{MSL}_n^{fin}) $, where $ \mr{MSL}_n^{fin}=\mr{Th}(n,n^2) $ are the finite approximation spaces for $ \mr{MSL}_{\sk} $. By a cofinality argument, we have that the same proof as in \textit{loc. cit.} works also if we use, as finite level approximation for $ \mr{MSL}_{\sk} $ ,the spaces $ \mr{MSL}_{2n}^{fin} $. So it turns out that $ \varphi $ is unique if:
\[ \lim\! {}^1\  \mr A^{2n-1}(\mr{MSL}_{2n}^{fin})=0 \]
Notice that it is enough to show the surjectivity of the maps:
\[ \ldots \rightarrow \mr A^{p}(\mr{MSL}_{2(n+1)}^{fin}) \stackrel{i_{n}^*}{\longrightarrow} \mr A^{p}(\mr{MSL}_{2n}) \rightarrow \ldots \]
\noindent induced by the maps of the direct system:
\[ \ldots \rightarrow \oocatname{T}(2n,4n^2)\stackrel{i_n}{\longrightarrow} \oocatname{T}(2(n+1), 4(n+1)^2) \rightarrow \ldots \]
But since $ A $ is  $ SL $-oriented, for every $ k $, we have isomorphisms:
\[ \mr A^{\bullet-2n}(SGr(2n,k)) \stackrel{\cup \mr{th}(2n,k)}{\longrightarrow} \mr A^{\bullet}(\Th{SGr(2n,k)}{\oocatname{T}(2n,k)}) \]
\noindent where $ \mr{th}(2n,k) $ denotes the Thom class of $ \oocatname{T}(2n,k) $. Using the last Thom isomorphism together with the computations in \cite[Theorem 9]{Ananyevskiy_PhD_Thesis}, we get that $ i_n^* $ are surjective and hence $ \lim^1 \mr A^{2n-1}(\mr{MSL}_{2n}^{fin})=0 $, giving us the uniqueness of $ \varphi $.
\end{proof}

Given an $ SL $-oriented ring spectrum $ \mb E $, with an $ SL $-orientation map $ \varphi: \mr{MSL} \rightarrow \mb E $, the $ \varphi $-induced $ SL $-orientation is uniquely determined. On the other hand, given an absolute $ SL $-orientation, its restriction to smooth schemes $ X \in \catname{Sm}_{\bigslant{}{S}} $ uniquely determines an associated $ SL $-orientation, thus we get the following:
\begin{co}\label{ch2:_abs_SL_or_on_eta_inv}
Let $ \mr A \in \mr{SH}(\sk) $ be an $ \eta $-invertible motivic ring spectrum. Then we have a one to one correspondence between the following data:
\begin{enumerate}
	\item $ SL $-orientations on $ \mr A $;
	\item maps of ring spectra $ \varphi: \mr{MSL}\longrightarrow \mr A $ such that $ \varphi(\mr{th}_{\mr{MSL}}(V))=\mr{th}_{\mr A}(V)   $ for any $ V $ vector $ SL $-bundle over $ X\in \catname{Sm}_{\bigslant{}{\sk}} $;
	\item absolute $ SL $-orientations.
\end{enumerate}
\end{co}

\begin{rmk}\label{ch2:_SL_and_Sp_absolute_orientaions}
We will only deal with $ SL_{\eta} $-oriented theories over smooth $ \sk $-schemes or just over some field $ \sk $. So from now on, with a slight abuse of notation, we will just say $ SL $-orientation instead of \textit{absolute} $ SL $-orientation. According to the corollary above, this will make no harm in the case we are working over a field. Using a Leray spectral sequence argument we can also extend \cref{ch2:_Abs_VS_normal_Orientations} to smooth $ \sk $-scheme $ S $ (cf. \cref{ch3:_Ana_Thm_9+}), but for most of our applications we will just work over a field, hence we will not need this result in such generality. Thanks to \cite[Remark 2.1.5]{Déglise_Fasel_Borel_Char}, we also need no distinction between $ Sp $-oriented and absolutely $ Sp $-oriented theories.
\end{rmk}

Recall from \cite{Panin-Walter} that there exists a spectrum $ \mr{BO}_{S} \in SH(S) $,  whenever $ \dfrac{1}{2} \in \mo{S}^{\times} $, that represents Hermitian K-theory\footnote{There are recent works towards possible extension to more general schemes where $ 2 $ is not invertible in the ring of regular functions. It is worth mentioning for example \cite{Kumar_PhD}.}. 
\begin{defn}
Let  $ \mr{KW}_{S}:=\mr{BO}_{S,\eta} $ be the Witt theory (absolute) spectrum defined by inverting the element $ \eta \in \mr{BO}^{-1,-1}_{S}(S) $ as done in detail in \cite[\S6 and Theorem 6.5]{Ananyevskiy_Witt_MSp_Reations}.  
\end{defn}
\begin{rmk}
The spectrum $ \mr{BO}_S $ is $ Sp $-oriented (cf. \cite{Panin-Walter}) and hence $ SL $-oriented. This induces an $ SL $-orientation on $ \mr{KW} $, and indeed $ \mr{KW} $ will be our main example and focus point as an $ SL_{\eta} $-oriented theory.
\end{rmk}

\subsection{Thom Isomorphism and Euler Classes}
For any $ SL $-oriented theory, we can then talk about Euler classes $ e(E,\theta) $ for $ SL $-bundles $ E $.

\begin{notation}
As already mentioned in the introduction, we will adopt the convention of \cite{Motivic_Euler_Char} for twisted cohomology theories. That means that given $ \mb E $ an $ SL $-oriented theory and $ L \rightarrow X $ a line bundle over some $ X \in \catname{Sch}_{\bigslant{}{S}} $, we denote the $ L $-twisted $ \mb E $- cohomology by:
\[ \mb E^{a,b}\left( X; L \right):=\mb E^{a+2,b+1}\left( \Th{X}{L} \right) \]
Similarly, given a vector bundle $ V\rightarrow X $, we will denote the $ L $-twisted $ \mb E $-cohomology on $ \Th{X}{V} $ as:
\[ \mb E^{a,b}\left( \Th{X}{V}; L \right):=\mb E^{a+2,b+1}\left( \Th{X}{V}\otimes \Th{X}{L} \right)\simeq \mb E^{a+2,b+1}\left( \Th{X}{V\oplus L} \right) \]
\end{notation}

\begin{rmk}
Notice that if $ L\simeq\A^1_X $ is the trivial line bundle, then $ \mb E^{a,b}(X;L):=\mb E^{a+2,b+1}(\pro^1_X)\simeq \mb E^{a,b}(X) $.
\end{rmk}

Given any vector bundle $ V $ of rank $ r $ on $ X \in \catname{Sch}_{\bigslant{}{S}} $, if $ L:=\det(V) $, we can construct the associated $ SL $-vector bundle given by $ V \oplus L^{-1} $ with its canonical trivialization of the determinant $ \omega_{can}: V\oplus L^{-1}\rightarrow \mo{X} $.

\begin{defn}\label{ch2:_def_twisted_Thom_class}
Let $ \catname{C} $ be a full subcategory of $ \catname{Sch}_{\bigslant{}{S}} $ and let $ \mb E \in \mr{SH}(S) $ be a ring spectrum with an $ SL $-orientation with respect to $ \catname{C} $. Let $ p: V \rightarrow X $ be a rank $ r $ vector bundle on $ X \in \catname{C} $ with determinant $ L:=\det(V) $.
\begin{enumerate}
	\item 	We define the Thom class in $ L^{-1} $-twisted cohomology by:
	\[ \mr{th}(V):=\mr{th}_{V\oplus L^{-1}}\in \mb E^{2r,r}(\Th{X}{V}; L^{-1}):=\mb E^{2r+2,r+1}(\Th{X}{V\oplus L^{-1}}) \]

	\item Let $ s_{0,L}: X \oplus L^{-1} \stackrel{s_0\oplus Id}{\longrightarrow} V\oplus L^{-1} $ be the map induced by the zero section $ s_0 $ of $ V $, then we define the (twisted) Euler class as:
	\[ e(E):=s_{0,L}^*\mr{th}^{\varphi}(V) \in \mb E^{2n+2,n+1}(\Th{X}{L^{-1}})=\mb E^{2n,n}(X; L^{-1}) \]
	
\end{enumerate}
\end{defn}

\begin{pr}[Twisted Thom Isomorphism]\label{ch2:_Twisted_Thom_SL}
Let $ p:V\rightarrow X $ be a rank $ r $ vector bundle over a scheme $ X $, and let $ \mb E\in \mr{SH}(S) $ be a $ SL $-oriented ring spectrum together with an $ SL $-orientation map $ \varphi $. Then we have an isomorphism:
\[ \vartheta_{V}^{\varphi}:= p^*(-) \cup \mr{th}^{\varphi}(V): \mb E^{*,*}(X; \det(V)) \longrightarrow \mb E^{*+2r,*+r}(\mr{Th}(V)) \]
\end{pr}

\begin{proof}
Denote by $ L:=\det(V) $ the determinant bundle of $ V $ and let $ \mc V $, $ \mc L $ be the locally free sheaves associated to $ V $ and $ L $. Using the absolute $ SL $-orientation induced by $ \varphi $, the construction in \cite[\S 3.10]{Levine_Raksit_Gauss_Bonnet} works verbatim in our case. Let us briefly sketch how one should proceed (more details can be found in \textit{loc. cit}).  The Thom class $ \mr{th}_{V\oplus L^{-1}}^{\varphi} $ gives us a Thom isomorphism:
\[ q^*(-)\cup \mr{th}_{V \oplus L^{-1}}^{\varphi}: \mb E^{\bullet,\bullet}(X)\longrightarrow \mb E^{\bullet+2(r+1),\bullet+r+1}(\Th{X}{V}; L^{-1}) \]
This means that we have an equivalence of spectra:
\[ \Sigma^{[\mo{}^{r+1}]-[\mc V]-[\mc L]}\mb E\simeq \mb E \]
Similarly we have $ \Sigma^{[\mo{}^{2}]-[\mc L]-[\mc L^{-1}]}\mb E\simeq \mb E $, and hence:
\begin{equation}\label{ch2:_eq_Abs_Twisted_Thom}
	\Sigma^{[\mo{}^{r}]-[\mc V]}\mb E\simeq \Sigma^{[\mo{}]-[\mc L]}\mb E
\end{equation}
The equivalence of \cref{ch2:_eq_Abs_Twisted_Thom} (together with homotopy invariance for $ p: V \rightarrow X $) gives us our isomorphism $ \vartheta_{V}^{\varphi} $.
\end{proof}

\begin{rmk}
If $ \mb E\in \mr{SH}(\sk) $ is an $ SL_{\eta} $-oriented motivic spectrum, we will drop the $ \varphi $ from the notation in virtue of \cref{ch2:_abs_SL_or_on_eta_inv}. Notice also that the Euler classes defined by the $ SL $-orientations will coincide, under the relevant Thom isomorphism, with the Euler classes defined in Chapter 1 using the formalism of  \cite{DJK}.
\end{rmk}

References with more details for Euler classes in $ SL $-oriented theories can be found in \cite[\S 3]{Ananyevskiy_SL_oriented} and \cite[\S 3]{Levine_Raksit_Gauss_Bonnet} (even if they work with $ SL $-orientations, everything can be adapted to our $ \varphi $-induced, absolute $ SL $-oriented case). A treatment of Euler classes closer to the one given here can also be found in \cite[\S 5]{Bachmann_Wickelgren}.

\begin{rmk}
Consider an $ SL $-oriented spectrum $ \mb E \in \mr{SH}(S) $, with an $ SL $-orientation map $ \varphi $, and $ V $ a vector bundle over $ X \in \catname{Sch}_{\bigslant{}{S}} $ with determinant $ L:=\det(V) $ (and associated locally free sheaves denoted by $ \mc V $ and $ \mc L $). Similarly to \cref{ch2:_BM_Thom_iso}, using $ L $-twisted Borel-Moore homology:
\[ \mb E^{\mr{BM}}_{a,b}\left( \bigslant{X}{S}; L \right):=\mb E^{\mr{BM}}_{a-2,b-1}\left( \bigslant{X}{S}, -[\mc L] \right) \]
\noindent we get $ L^{-1} $-twisted Thom isomorphism:
\[ \mb E^{\mr{BM}}_{a-2r,b-r}\left( \bigslant{X}{S}; L^{-1} \right)\stackrel{\sim}{\longrightarrow}  \mb E^{\mr{BM}}_{a,b}\left( \bigslant{X}{S}, [V]\right)  \]
\noindent using the $ \varphi $-induced Thom classes $ \mr{th}^{\varphi}(V)\in \mb E^{2r,r}(X; L^{-1}) $. To remember how twisted Thom isomorphism works (both for cohomology and Borel-Moore homology), it is enough to remember that:
\[ \Sigma^{[\mc V]\oplus [\mc L^{-1}]}\mb E\simeq \Sigma^{2r+2,r+1}\mb E \]
\noindent or equivalently:
\[ \Sigma^{-2(r+1), -(r+1)}\Sigma^{[\mc V]}\mb E\simeq \Sigma^{2,1}\Sigma^{-[\mc L^{-1}]}\mb E \]
\end{rmk}

\begin{construction}\label{ch2:_tautological_symbol}
We will now construct a \textit{symbol} element associated to a section of a line bundle, using the construction of a symbol associated to an invertible function on a scheme $ X $ as done in \cite[Definition 6.1]{Ananyevskiy_SL_oriented}. For simplicity we will restrict to the case of a scheme, but the same procedure will work for any NL-stack without changing a word. Recall from \textit{loc. cit.} that given $ u \in \Gamma(X, \mo{X}^{\times}) $, for $ X \in \catname{Sm}_{\bigslant{}{S}} $ and $ \mb E \in \mr{SH}(X) $, we have a well defined element $ \langle u\rangle \in \mb E^{0,0}(X) $ induced by the multiplication by $ u $ on $ T=\bigslant{\A^1_X}{\mb G_{m,X}} $. This element $ \langle u \rangle $ is  called the \textit{symbol} associated to $ u $. Consider now a line bundle $ p:L\rightarrow X $, and consider $ \lambda: X \rightarrow L $ a section. Denote by $ \mc Z(\lambda) $ the vanishing locus of $ \lambda $:
\begin{center}
	\begin{tikzpicture}[baseline={(0,1)}, scale=1.5]
		\node (a) at (0,1) {$ \mc Z(\lambda) $};
		\node (b) at (1, 1) {$ X $};
		\node (c)  at (0,0) {$  X $};
		\node (d) at (1,0) {$ L $};
		\node (e) at (0.25,0.75) {$ \ulcorner $};
		\node (f) at (0.5,0.5) {$  $};

		\path[font=\scriptsize,>= angle 90]
		
		(a) edge [closed] node [above ] {$ \iota_{\lambda} $} (b)
		(a) edge [closed] node [left] {$  $} (c)
		(b) edge[closed] node [right] {$ s_0 $} (d)
		(c) edge [closed] node [below] {$ \lambda $} (d);
	\end{tikzpicture}
\end{center}
\noindent Let $ j_{\lambda}:U(\lambda)\into X $ be the open complement in $ X $ of $ \mc Z(\lambda) $. Then $ \lambda $ induces a non vanishing section $ j^*\lambda: U \rightarrow j^*_{\lambda}L^{\times} $ of $ j^*_{\lambda}L $. But this means we can trivialise $ j_{\lambda}^*L $, i.e. we have:
\[ \tau_{j^*_{\lambda}}: \mb A^{1}_{U(\lambda)}\stackrel{\sim}{\longrightarrow} j^*_{\lambda}L \]
\noindent with associated inverse:
\[ \left( \tau_{j^*_{\lambda}} \right)^{-1}: j^*_{\lambda}L \stackrel{\sim}{\longrightarrow} \A^1_{U(\lambda)} \]
Taking the associated Thom spaces, we get:

\[ \mr{Th}(\tau_{j^*_{\lambda}}^{-1}):\Th{U(\lambda)}{j^*_{\lambda}L}\simeq\Sigma^{j^*_{\lambda}\mc  L} \mbbm 1_{U(\lambda)} \longrightarrow \Th{U(\lambda)}{\A^1}\simeq \Sigma^{\mo{}}\mbbm 1_{U(\lambda)} \]
Twisting by $ \Sigma^{-\mo{}} $, we have:

\[  \Sigma^{-\mo{}}\mr{Th}(\tau_{j^*_{\lambda}}^{-1}): \Sigma^{-\mo{}}\Sigma^{j^*_{\lambda}\mc  L} \mbbm 1_{U(\lambda)} \longrightarrow \mbbm 1_{U(\lambda)} \]
\begin{defn}\label{ch2:_symbol_of_line_bundle_section}
	With the same notation above, let $ \mb E\in \mr{SH}(S) $ be a ring spectrum with unit $ u: \mbbm 1 \rightarrow \mb E $, then we define the $ \mb E $-\textit{symbol} associated to $ \lambda: X \rightarrow L $ to be:
	\[ \langle \lambda \rangle_{\mb E}:=u\circ \Sigma^{-\mo{}}\mr{Th}(\tau_{j^*_{\lambda}}^{-1}): \Sigma^{-\mo{}}\Th{U(\lambda)}{j^*_{\lambda}L} \rightarrow \mb E \in \mb E^{0,0}(U(\lambda); j^*_{\lambda}L) \]
	
\end{defn}

\begin{exa}
	Consider the section of $ \mo{\pro^2}(2) $ given by $ Q=T_1^2-4T_0T_2 $, then $ U(Q)=\pro^2\setminus C $ with $ C $ the zero-locus of $ Q $. Then for any $ \mb E\in \mr{SH}(S) $ we have:
	\[ \langle Q\rangle \in \mb E^{0,0}(\pro^2\setminus C; \mo{}(2)) \]
	Suppose that $ \mb E $ is either an element of $ \mr{SH}(S)[\eta^{-1}] $ or it is $ SL $-oriented. Then for any scheme $ X $ and any line bundle $ L $ over $ X $, by \cite[Proposition 3.3.1]{Haution_Odd_VB} or \cite[Theorem 4.3]{Ananyevskiy_SL_oriented} respectively, there exists an isomorphism:
	\[ \varphi: \mb E(X; L^{\otimes 2}) \stackrel{\sim}{\longrightarrow} \mb E(X) \]
	Hence we get a well defined element:
	\[ q_0:=\varphi(\langle Q\rangle)\in \mb E^{0,0}(\pro^2\setminus C) \]
\end{exa}

\begin{exa}
	Consider $ p:L\rightarrow X $ a line bundle over $ X $. Let:
	\[ t_{can}: L\rightarrow p^*L \]
	\noindent be the tautological section. Then $ U(t_{can})=L^{\times}=L\setminus 0 $ and for any $ \mb E\in \mr{SH}(S) $ we get:
	\[ \langle t_{can}\rangle \in \mb E^{0,0}(L^{\times}; L) \]
	Consider  $ X=BGL_n $ and $ L=\mo{}(1) $. Then $ L^{\times}\simeq BSL_n $ and we get:
	\[ \langle t_{can}\rangle \in \mb E^{0,0}(BSL_n; \mo{}(1)) \]
	Sometimes we will refer to $ \langle t_{can}\rangle $ as the \textit{tautological symbol} associated to $ L $.
\end{exa}

\end{construction}

%
%

\subsection{$ SL $-Orientations for NL-Stacks}\label{ch2:_orientations_Stacks}
We will now present an easy way to get Thom classes and Euler classes on NL-stacks. The methods used here can be adapted to most of the common $ G $-orientations used in the literature, but since we will need to specialise to $ SL $-oriented spectra anyway, we will only talk about those.\\

Consider $ \oocatname{U_n}\rightarrow \mc BSL_n $ the universal bundle over $ \mc BSL_n $ (the one corresponding, under Yoneda, to the identity map of $ \mc BSL_n $. 

%
%

\begin{pr}\label{ch2:_Univ_Thom_Iso}
Let $ \mb E \in \mr{SH}(S) $ be an $ SL $-oriented ring spectrum. Then we have a natural equivalence of mapping spectra:
\[ \tau: \mb E(\mc BSL_n)\longrightarrow \Sigma^{2n,n}\mb E(\Th{\mc BSL_n}{\oocatname{U_n}}) \]
\end{pr}

\begin{proof}
First of all we need to construct the map $ \tau $ and then we will prove it is indeed an isomorphism. Let $ SGr_S(j,k) $ be the special linear Grassmannian, with tautological bundle $ \oocatname{T}(j,k) $. For each double index $ (j,k) $, we have natural maps $ \sigma_{j,k}; SGr_S(j,k)\rightarrow \mc BSL_n $ classifying the tautological bundles, that is, we have cartesian squares:
\begin{center}
	\begin{tikzpicture}[baseline={(0,0)}, scale=1.5]
		\node (a) at (0,1) {$  \oocatname{T}(j,k) $};
		\node (b) at (2, 1) {$ \oocatname{U_n} $};
		\node (c)  at (0,0) {$  SGr_S(j,k) $};
		\node (d) at (2,0) {$ \mc BSL_n $};
		\node (e) at (0.2,0.758) {$ \ulcorner $};
		\node (f) at (0.5,0.5) {$  $};

		\path[font=\scriptsize,>= angle 90]
		
		(a) edge [->] node [above ] {$  $} (b)
		(a) edge [->] node [left] {$  $} (c)
		(b) edge[->] node [right] {$  $} (d)
		(c) edge [->] node [below] {$ \sigma_{j,k} $} (d);
	\end{tikzpicture}
\end{center}
By \cite[Proposition 3.33]{Motivic_Vistoli}, we have a natural equivalence:
\[ \beta_{\infty}:  \colim{m} \pi_{SGr_S(n,m)} \epfs \Th{SGr_S(n,m)}{\oocatname{T}(n,m)}\stackrel{\sim}{\rightarrow} \pi_{\mc BSL_n}\epfs \Th{\mc BSL_n}{\oocatname{U_n}} \]
But the left hand side is by definition $ \mr{MSL}_n:=\colim{m} \pi_{SGr_S(n,m)} \epfs \Th{SGr_S(n,m)}{\oocatname{T}(n,m)}   $ as defined in \cite[\S 4]{Panin-Walter_MSL_MSp}. By construction of $ \mr{MSL} $ as a spectrum (cf. \cite[\S 4]{Panin-Walter_MSL_MSp}), we have a natural maps:
\[ u_n: \Sigma^{-2n,-n}\mr{MSL}_n \rightarrow \mr{MSL} \]
By \cite[Theorem 5.9]{Panin-Walter_MSL_MSp} we have a map of motivic spectra $ \varphi: \mr{MSL}\rightarrow \mb E $. Consider the following composition of maps:

\[ \mr{th}_{\oocatname{U}_n}:=\Sigma^{2n,n}(\varphi\circ u_n)\circ \beta_{\infty}^{-1}: \pi_{\mc BSL_n}\epfs \Th{\mc BSL_n}{\oocatname{U_n}}\longrightarrow \Sigma^{2n,n}\mb E \]

This means that we have an element $ \mr{th}_{\oocatname{U}_n}\in \mb E^{2n,n}(\Th{\mc BSL_n}{\oocatname{U}_n}) $. Let $ p_n: \oocatname{U}_n\rightarrow \mc BSL_n $ be the projection map and let us finally define the map we are looking for:
\[ \begin{array}{cccc}
	\tau: & \mb E^{\bullet,\bullet}(\mc BSL_n) & \longrightarrow & \mb E^{\bullet+2n,\bullet+n}(\Th{\mc BSL_n}{\oocatname{U}_n})\\
	& x & \mapsto & \mr{th}_{\oocatname{U}_n}\cup p^*x
\end{array} \]

Rewriting $ \mb E(\mc BSL_n) $ and $ \Sigma^{2n,n}\mb E(\Th{\mc BSL_n}{\oocatname{U_n}}) $   in terms of mapping spectra, we want to show that the map:
\[ \tau: \iMap_{\mr{SH}(\mc BSL_n)}(\mbbm 1_{\mc BSL_n}, \pi_{\mc BSL_n}^*\mb E)\longrightarrow \iMap_{\mr{SH}(\mc BSL_n)}(\Th{\mc BSL_n}{\oocatname{U}_n}, \Sigma^{2n,n}\pi_{\mc BSL_n}^*\mb E)   \]

\noindent is indeed an equivalence. Since $ SL_n $ is special, we can take $ a: S \rightarrow \mc BSL_n $ as our NL-atlas. By \cite[Lemma 4.1.1]{Chowdhury24}, the map $ a^*: \mr{SH}(\mc BSL_n)\rightarrow \mr{SH}(S) $ is conservative, hence $ \tau $ is an equivalence if and only if $ a^*\tau $ is an equivalence. Notice that $ a^*\oocatname{U}_n\simeq \A^n_S $ and denote by $ q_n:=a^*p_n: \A^n_S\rightarrow S $ the projection map. Then we have:
\[ a^*\tau(-)=a^*\mr{th}_{\oocatname{U}_n}\cup q_n^*-=\mr{th}_{\A^n_S}\cup q_n^*- \]
But this means that $ a^*\tau $ is the Thom isomorphism map associated to $ \A^n_S $, so it is indeed an equivalence as we wanted to show. 
\end{proof}

\begin{defn}
We define the \textit{canonical} Thom class of $ \oocatname{U_n} \rightarrow \mc BSL_n $ as the element:
\[ \mr{th}(\oocatname{U_n}):=\tau(1_{\mc BSL_n}) \in \mb E^{2n,n}(\Th{\mc BSL_n}{\mc U_n}) \]
\noindent where $ 1_{\mc BSL_n} \in\mb E^{0,0}(\mc BSL_n) $ is the identity element in the $ \mb E $-cohomology of $ \mc BSL_n $.
\end{defn}

Now, let $ \mc X \in \oocatname{ASt}_{\bigslant{}{S}}^{NL} $ be a NL-stack. Let $ v:  V \rightarrow \mc X $ be a vector bundle of rank $ n $ with trivialised determinant. The vector bundle $  V $ is classified by a map $ f_V $ such that:
\begin{center}
\begin{tikzpicture}[baseline={(0,0)}, scale=1.5]
	\node (a) at (0,1) {$  V $};
	\node (b) at (2, 1) {$ \oocatname{U_n} $};
	\node (c)  at (0,0) {$  \mc X $};
	\node (d) at (2,0) {$ \mc BSL_n $};
	\node (e) at (0.2,0.8) {$ \ulcorner $};
	\node (f) at (0.5,0.5) {$  $};

	\path[font=\scriptsize,>= angle 90]
	
	(a) edge [->] node [above ] {$  $} (b)
	(a) edge [->] node [left] {$  $} (c)
	(b) edge[->] node [right] {$  $} (d)
	(c) edge [->] node [below] {$ f_{ V} $} (d);
\end{tikzpicture}
\end{center}

\begin{defn}
We define the Thom class of the special linear vector bundle $  V \rightarrow \mc X $ with values in a $ SL $-oriented ring spectrum $ \mb E\in \mr{SH}(S) $ as:
\[ \mr{th}({ V}):= f_{ V}^*\mr{th}({\oocatname{U_n}})\in \mb E^{2n,n}\left(  \Th{\mc X}{ V} \right) \]
\noindent where $ f_{ V}: \mc X \rightarrow \mc BSL_n $ is the map classifying the special linear bundle $  V $.
\end{defn}

As we did for schemes, once we have Thom classes for vector bundles with trivialised determinants, we can define Thom classes for general vector bundles living in twisted cohomology. Indeed if $  V\rightarrow \mc X $ is a vector bundle of rank $ n $ with determinant $  L:=\det( V) $, then $  V \oplus  L^{-1} $ is a vector bundle of rank $ n+1 $ with trivial determinant. We can define:

\begin{defn}
The Thom class for $  V $ in $  L^{-1} $-twisted $ \mb E $-cohomology is the element:
\[ \mr{th}_{ V}:=\mr{th}_{ V\oplus  L^{-1}}\in \mb E^{2n,n}(\Th{ \mc X}{ V};  L^{-1}):=\mb E^{2n+2,n+1}\left( \Th{\mc X}{ V\oplus  L^{-1}} \right) \]
\end{defn}

\begin{defn}
In the same notation as above, we defined the $ \mb E $-valued Euler class of a vector bundle $  V \rightarrow \mc X $ of rank $ n $ as:
\[ e( V):=s^*\mr{th}_{ V} \in \mb E^{2n, n}(X;  L^{-1}):=\mb E^{2n+2,n+1}(\Th{\mc X}{ L^{-1}}) \]
\noindent where $ s^* $ is the pullback map induced by the zero section $ s_0: \mc X \rightarrow  V $.
\end{defn}

\section{$ SL_{\eta} $-Theories on $ BGL_n $}

Recall from \cite[\S4]{Motivic_Euler_Char} that for any $ SL $-oriented ring spectrum $ \mb E $ we have a map:
\begin{equation}\label{ch3:_eq_Marc_map}
	\pi^*: \mr A^{\bullet}\left( {B}GL_n \right) \oplus \mr A^{\bullet}\left( {B}GL_n; det(E_n) \right) \longrightarrow \mr A^{\bullet}\left( {B}SL_n \right) 
\end{equation}

Indeed we have the pullback map:
\[ \pi^*_0: \mr A^{\bullet}\left( {B}GL_n \right)\longrightarrow \mr A^{\bullet}\left( {B}SL_n \right)   \]
\noindent induced by $ \pi_0:  BSL_n \longrightarrow  BGL_n $. Now consider the tautological rank $ n $ vector bundle $ E_n \rightarrow  BGL_n $ with $ \det(E_n)=\mo{}(1) $, then its pullback $ \pi_0^*E_n $ will be the tautological special linear bundle over $  BSL_n $, so we have a canonical trivialization $ \theta: \det\left( \pi_0^*E_n \right)\simeq \pi_0^*\mo{}(1)\stackrel{\sim}{\rightarrow} \mo{ BSL_n} $. Composing the pullback map on the twisted theories with $ \theta_*:=(\tau_{\pi_0^*\mo{}(1)})^{-1} $,  given by the inverse Thom isomorphism (cf. \ref{ch2:_Def_Thom_iso}), we get:

\[ \mr A^{\bullet}\left( {B}GL_n; \mo{}(1) \right)\stackrel{\pi_0^*}{\longrightarrow} \mr A^{\bullet}\left( {B}SL_n ; \pi_0^*\mo{}(1)\right) \stackrel{\theta_*}{\longrightarrow} \mr A^{\bullet}\left( {B}SL_n \right)   \]
\noindent and we denote this map as $ \pi_1^*:=\theta_* \circ \pi_0^* $. Then putting together $ \pi_0^* $ and $ \pi_1^* $ we get our desired $ \pi^* $.

We would like to reduce the computations of characteristic classes from general vector bundles to special linear ones. With a minor adaptation of the arguments in \cite[Proposition 4.1]{Motivic_Euler_Char} we can prove the following:

\begin{pr}\label{ch3:_4.1_MEC}
	Let $ S\in \catname{Sm}_{\bigslant{}{\sk}} $. Let $ E_n $ be the universal tautological bundle of $ BGL_n $, and let $ \mr A\in \mr{SH}(S) $ be an $ SL $-oriented ring spectrum. Then the map \eqref{ch3:_eq_Marc_map}:
	\[ \pi^*: \mr A^{\bullet}\left( {B}GL_n \right) \oplus \mr A^{\bullet}\left( {B}GL_n; det(E_n) \right) \longrightarrow \mr A^{\bullet}\left( {B}SL_n \right)  \]
	\noindent is an isomorphism. In particular we have:
	\begin{enumerate}
		\item [($ n=2m $)]  \[ \pi^*\left( \mr A^{\bullet}\left( {B}GL_n \right) \right)\simeq \mr A^{\bullet}(S)\llbracket p_1,\ldots, p_{m-1} , e^2 \rrbracket  \]
		\[ \pi^*\left( \mr A^{\bullet}\left( {B}GL_n; \det(E_n) \right) \right)\simeq e\cdot \mr A^{\bullet}(S)\llbracket p_1,\ldots, p_m , e^2 \rrbracket  \]
		\noindent where $ p_i=p_i(\oocatname{T}(n,\infty)) $ are the Pontryagin classes (cf. \cite[Theorem 10]{Ananyevskiy_PhD_Thesis}) of to the tautological bundle $ \oocatname{T}(n,\infty) $ over $ BSL_n $ and $ e=e(\oocatname{T}(n,\infty)) $ is the Euler class of the tautological bundle.
		\item [($ n=2m+1 $)] Then $ \mr A^{\bullet}\left( {B}GL_n; \det(E_n)\right) \simeq 0 $ and:
		\[ \pi^*\left( \mr A^{\bullet}\left( {B}GL_n \right) \right)\stackrel{}{\longrightarrow} \mr A^{\bullet}\left( {B}SL_n \right)   \]
		\noindent is an isomorphism.
		
	\end{enumerate}
	
\end{pr}
\begin{proof}
	Let us consider the line bundle $ q:\mo{}(1):=\det(E_n) \longrightarrow {B}GL_n $. We can identify the map $ \pi_0:{B}SL_n \longrightarrow {B}GL_n $ with the $ \mb G_m $-principal bundle $ q\circ j:\mo{}(1)\setminus \set{0} \longrightarrow {B}GL_n $ where $ j: \mo{}(1) \setminus \set{0}\into \mo{}(1) $. By homotopy invariance we can also identify $ \mr A({B}GL_n)\simeq \mr A(\mo{}(1)) $, thus we have the localization sequences:
	\begin{align}\label{ch3:_eq_BGL_Untwisted}
		\begin{split}
			\ldots \rightarrow \mr A^{a,b}\left( {B}GL_n \right)  \stackrel{}{\rightarrow} &\mr A^{a,b}\left( {B}SL_n \right) -\!\!\!-\!\!\!-\ldots  \\
			\ldots&\stackrel{\partial_{a,b}}{\rightarrow} \mr A^{a-1,b-1}\left( {B}GL_n; det(E_n) \right)  \stackrel{e(\det(E_n))\cup \cdot}{\rightarrow} \mr A^{a+1,b}\left( {B}GL_n \right)  \rightarrow \ldots
		\end{split}
	\end{align}
	\begin{equation}
		\ldots \rightarrow \mr A^{a,b}\left( {B}GL_n; \mo{}(-1) \right)  \stackrel{}{\rightarrow} \mr A^{a,b}\left( {B}SL_n; \pi_0^*\mo{}(-1)\right)  \stackrel{\bar\partial_{a,b}}{\rightarrow} \mr A^{a-1,b-1}\left( {B}GL_n\right)  \stackrel{}{\rightarrow} \ldots 
	\end{equation}

	\noindent By \cite[lemma 4.3]{Motivic_Euler_Char} the cup product with $ e(L) $ for any line bundle $ L $ is zero in any $ \eta $-inverted $ SL $-oriented ring spectrum, so the long exact sequence \eqref{ch3:_eq_BGL_Untwisted} splits in short exact sequences: 
	
	\[ 0 \rightarrow \mr A^{a,b}\left( {B}GL_n \right)  \stackrel{}{\rightarrow} \mr A^{a,b}\left( {B}SL_n \right)  \stackrel{\partial_{a,b}}{\rightarrow} \mr A^{a-1,b-1}\left( {B}GL_n; det(E_n) \right)  \rightarrow 0 \]
	
	\noindent We want now to find a splitting for $ \partial $ and we already have a natural candidate $ \pi_1^* $, in the notation used above the proposition. Let us consider the tautological section $ can: \mo{}(1)\rightarrow q^*\mo{}(1) $. Using \cref{ch2:_symbol_of_line_bundle_section}, the tautological section defines an element $ \langle t_{can} \rangle \in \mr A^{0,0}({B}SL_n, \pi^*_0\mo{}(1))\simeq  \mr A^{0,0}({B}SL_n, \pi^*_0\mo{}(-1)) $. The map $ \pi_1^* $ we constructed before was:
	
	\[ \mr A^{\bullet}\left( {B}GL_n; \mo{}(1) \right)\stackrel{\pi_0^*}{\longrightarrow} \mr A^{\bullet}\left( {B}SL_n ; \pi_0^*\mo{}(1)\right) \stackrel{\theta_*}{\longrightarrow} \mr A^{\bullet}\left( {B}SL_n \right)  \]
	
	\noindent but $ \theta_* $ is just multiplication by $ \langle t_{can} \rangle $.  To show that the short exact sequence actually splits we need to prove that $ \partial \circ \pi_1^*$ is an isomorphism. As a first step we claim that for any $ x \in \mr A^{\bullet}({B}GL_n; \mo{}(1)) $ we have:
	\[ \partial(\pi_1^*(x))=\bar\partial(\langle t_{can} \rangle)\cup x \]
	\noindent Indeed:
	\[ \partial(\pi_1^*(x))=\partial( \langle t_{can} \rangle \cup \pi_0^*(x) )= \partial( \langle t_{can} \rangle \cup q^*j^*(x) )  \]
	\noindent and $ \partial $ is $ \mr A^{\bullet}({B}GL_n; \mo{}(1)) $-linear, with $ \mr A^{\bullet}({B}GL_n; \mo{}(1))  $ acting on $ E^{\bullet}({B}SL_n)  $ via multiplication through $ q^*j^* $ and on $ E^{\bullet}({B}GL_n; \mo{}(1))  $ just via multiplication, and this proves the claim. \\
	We want now to prove that multiplication by $ \bar\partial(\langle t_{can} \rangle) \in \mr A^{-1,-1}({B}GL_n) $ is an isomorphism:
	\[ \bar \partial(\langle t_{can} \rangle) \cup \cdot : \mr A^{a,b}({B}GL_n; \mo{}(1)) \longrightarrow \mr A^{a-1,b-1}({B}GL_n; \mo{}(1)) \]
	Using a Mayer-Vietoris argument on the finite approximation pieces $ {B}_mGL_n $ we can reduce ourselves to a local computation, where we already know the result by \cite[Lemma 6.4]{Ananyevskiy_SL_oriented}. Indeed given two open sets $ U, V $ of $ {B}_mGL_n $, we get a Nisnevich excision square related to $ \set{U,V, U\cup V} $ and a Mayer-Vietoris exact sequence for any Nisnevich sheaf. In particular multiplication by an element $ u \in \mr A^{-1,-1}({B}_mGL_n) $ gives us a map of Mayer-Vietoris sequences:
	
	\begin{center}
		\begin{tikzpicture}[baseline={(0,0)}, scale=2]
			\node (a5) at (0,5) {$ \ldots $};
			\node (b5) at (4, 5) {$ \ldots $};
			
			\node (a4)  at (0,4) {$ \mr A^{a,b}(U\cup V; \mo{}(1)) $};
			\node (b4) at (4,4) {$  \mr A^{a-1,b-1}(U\cup V; \mo{}(1))  $};
			
			\node (a3) at (0,3) {$ \mr A^{a,b}(U; \mo{}(1)) \oplus \mr A^{a,b}(V; \mo{}(1))  $};
			\node (b3) at (4,3) {$ \mr A^{a-1,b-1}(U; \mo{}(1)) \oplus \mr A^{a-1,b-1}(V; \mo{}(1))   $};
			
			\node (a2) at (0,2) {$ \mr A^{a,b}(U\cap V; \mo{}(1))  $};
			\node (b2) at (4,2) {$ \mr A^{a-1,b-1}(U\cap V; \mo{}(1))  $};
			
			\node (a1) at (0,1) {$ \mr A^{a+1,b}(U\cup V; \mo{}(1))  $};
			\node (b1) at (4,1) {$ \mr A^{a,b-1}(U\cup V; \mo{}(1))  $};
			
			\node (a0) at (0,0) {$ \ldots $};
			\node (b0) at (4,0) {$ \ldots  $};

			\path[font=\scriptsize,>= angle 90]
			
			
			(a5) edge [->] node [left ] {$  $} (a4)
			(a4) edge [->] node [left] {$  $} (a3)
			(a3) edge [->] node [left] {$  $} (a2)
			(a2) edge [->] node [left] {$  $} (a1)
			(a1) edge [->] node [left] {$  $} (a0)
			(b5) edge [->] node [right ] {$  $} (b4)
			(b4) edge [->] node [right] {$  $} (b3)
			(b3) edge [->] node [right] {$  $} (b2)
			(b2) edge [->] node [right] {$  $} (b1)
			(b1) edge [->] node [right] {$  $} (b0)
			
			(a4) edge [->] node [above ] {$ u \cup \cdot $} (b4)
			(a3) edge [->] node [above] {$ (u\cup \cdot, u\cup \cdot ) $} (b3)
			(a2) edge[->] node [above] {$ u\cup \cdot $} (b2)
			(a1) edge [->] node [above] {$ u\cup \cdot $} (b1);
		\end{tikzpicture}
	\end{center}
	So to show that multiplication by $ \bar\partial(\langle t_{can}\rangle) $ is an isomorphism, it is enough to show that $ \bar\partial(\langle t_{can}\rangle)  $ restricts to an invertible element when passing to open sets $ U_i $ covering $ {B}_mGL_n $. But in local coordinates $\bar \partial(\langle t_{can}\rangle)  $ restricts to $ \eta $ by \cite[Lemma 6.4]{Ananyevskiy_SL_oriented} and thus we get our desired splitting for $ B_mGL_n $. This implies that for any $ k $ we have:
	\[ \pi_m^*: \mr A^{k}(B_mGL_n)\oplus \mr A^{k}(B_mGL_n; \mo{}(1))\stackrel{\sim}{\longrightarrow} \mr A^{k}(B_mSL_n) \]
	\noindent where $ \pi_m^* $ is the map induced by $ \pi^* $ at the finite level approximations. Since the isomorphism holds for any $ k $, we have an equivalence of mapping spectra:
	
	\[ \pi^*_m:  \mr A(B_mGL_n)\oplus \mr A^{}(B_mGL_n; \mo{}(1))\stackrel{\sim}{\longrightarrow} \mr A^{}(B_mSL_n)  \]
	\noindent where $ \mr A(B_mG)=\iMap(\Sigma^{\infty} B_mG, \mr A) $ (and similarly for the twisted version). But since $ \mr A(BG)=\iMap(BG, \mr A)=\lim_m \iMap(B_mG, \mr A) $, from the equivalence of mapping spectra $ \pi_m^* $, we get:
	\[ \pi^*: \mr A(BGL_n)\oplus \mr A^{}(BGL_n; \mo{}(1))\stackrel{\sim}{\longrightarrow} \mr A^{}(BSL_n)  \]
	\noindent again at the level of mapping spectra, and this gives us our claim.
	
	For the explicit presentation of the image of $ \pi^* $ the same argument as in the proof of \cite[Proposition 4.1]{Motivic_Euler_Char} will give us the result, using the statement of \cite[Theorem 10]{Ananyevskiy_PhD_Thesis} generalised over any smooth $ \sk $-scheme $ S $ in \cref{ch3:_Kunneth_BSL_n} (we will actually generalised it to smooth NL-stacks).
	
\end{proof}

\begin{rmk}
	The previous proposition was already proved independently in \cite[Remark 6.3.7]{Haution_Odd_VB} using stronger results. Indeed in \textit{loc. cit.} it is proved that $ BGL_{2r}\simeq BGL_{2r+1} $ and $ BGL_{2r+1}\simeq BSL_{2r+1} $ in $ \mr{SH}(S)[\eta^{-1}] $ (this are respectively \cite[Theorem 6.3.3, Theorem 6.3.6]{Haution_Odd_VB}). 
\end{rmk}

The proof of \cref{ch3:_4.1_MEC}, used as a key input that $ BSL_n $ can be seen as the complement of the zero section of the determinant tautological bundle of $ BGL_n $. And that is all we actually need: the same proof goes through verbatim if we consider $ BN \rightarrow BN_G $ where $ N_G $ is the normaliser of the torus inside $ GL_2 $. The model for $ BN_G $ is given by $\left( \quot{GL_2}{N_G}\right)\times^{GL_2} EGL_2 $, where again $ \quot{GL_2}{N_G}\simeq \quot{SL_2}{N}\simeq \pro^2\setminus C $. We then have:

\begin{pr}\label{ch3:_Norm_GL}
	Let $ S\in \catname{Sm}_{\bigslant{}{\sk}} $. For any $ SL_{\eta} $-oriented ring spectrum $ \mr A\in \mr{SH}(S) $ we have an induced isomorphism:
	\[ \pi^*: \mr A^{\bullet}\left( {B}N_G \right) \oplus \mr A^{\bullet}\left( {B}N_G; \mo{}(1)\right) \longrightarrow \mr A^{\bullet}\left( {B}N \right)  \]
\end{pr}
\begin{proof}
	Use the same exact proof as in \cref{ch3:_4.1_MEC}, replacing $ BSL_n $ with $ BN $ and $ BGL_n $ with $ BN_G $ and pulling back the appropriate maps along the morphisms $ BN \rightarrow BSL_2 $ and $ BN_G \rightarrow BGL_2 $ induced by $ N\into SL_2 $ and $ N_G \into GL_2 $ respectively.
\end{proof}
\begin{rmk}
	The previous proposition was also proved in \cite[Proposition 2.5.10]{Viergever_PhD} under the assumption that the unit map of $ \mr A $ makes the $ \mr A^{\bullet}(-) $-cohomology into a module for the sheaf of Witt groups $ H\mc W $.
\end{rmk}
\begin{rmk}
	As a corollary of \cref{ch3:_Norm_GL}, it is not hard to get an additive description for $ \mr A^{\bullet}(BN_G) $, using \cite[Proposition 3.18]{Witt_Loc_PhD}, and a multiplicative description of $ \mr{KW}^{\bullet}(BN_G) $ using \cite[Corollary 3.26]{Witt_Loc_PhD}. We are very  grateful to A.\! Viergever that made us realise we could improve and apply our results to the case of $ BN_G $.
\end{rmk}

\section{K\"unneth Formulas}

We want now to prove some K\"unneth formulas for $ \mc BSL_n, \mc BGL_n$. Let us start with an extension of \cite[Theorem 9]{Ananyevskiy_PhD_Thesis}. Notice that using our model for the approximations $ B_mSL_n $ we have that $ B_mSL_n\simeq SGr(n, n+m) $. Since we will work over different base schemes we will denote the special Grassmannian over some scheme $ S $ as $ SGr_{S}(n,k) $.

\begin{pr}\label{ch3:_Ana_Thm_9+}
	Let $ S \in \catname{Sm}_{\bigslant{}{\sk}} $ and let $ \mr A \in \mr{SH}(S) $ be an $ SL_{\eta} $-oriented ring spectrum. Let $ T_1:=\oocatname{T}(2n,2k+1) $ be the tautological bundle of rank $ 2n $ on $ SGr_{\sk}(2n, 2k+1) $, with Pontryagin classes $ p_i(T_1) $. Then there is an isomorphism of $ \mr A^{\bullet}(S) $-algebras:
	\[ \varphi: \quot{\mr A^{\bullet}(S)[p_1,p_2,\ldots, p_k, e]}{J_{2n, 2k+1}} \longrightarrow \mr A^{\bullet}( SGr_S(2n,2k+1))  \]
	\noindent where:
	\[ J_{2n,2k+1}:=\left( e^2-p_n, g_{k-n+1}(p_1,\ldots, p_n), g_{k-n+2}(p_1,\ldots,p_n), \ldots, g_k(p_1,\ldots,p_m) \right) \]
	\noindent is defined with the same polynomials as in \cite[Theorem 9]{Ananyevskiy_PhD_Thesis} and $ \varphi $ is defined by sending $ \varphi(p_i):=p_i(T_1) $ and $ \varphi(e):=e(T_1) $.
\end{pr}

\begin{proof}
	If $ S=\spec(\sk) $, then this is just \cite[Theorem 9]{Ananyevskiy_PhD_Thesis}. For a general smooth $ \sk $-scheme $ S $, we have a map:
	\[ {\varphi}:\mr A^{\bullet}(S)[p_1,\ldots, e] \longrightarrow \mr A^{\bullet}(SGr_S(2n,2k+1)) \]
	\noindent sending $ p_i $ to $ p_i(T_1) $ and $ e$ to $ e(T_1) $.
	The structure map $ \pi_S: S \rightarrow \spec(\sk) $ induces a pullback map:
	\[ \pi_S^*:\mr A^{\bullet}(SGr_{\sk}(2n,2k+1))\rightarrow \mr A^{\bullet}(SGr_S(2n,2k+1)) \]
	The tautological bundle $ T_1 $ on $ SGr_S(2n,2k+1) $ is the pullback of the tautological bundle $ T_{1, \sk} $ on $ SGr_{\sk}(2n,2k+1) $, hence the Pontryagin and Euler classes of $ T_{1} $ satisfy all the relations in $ J_{2m,2k+1} $ since they hold for the classes of $ T_{1,\sk} $ (by \cite[Theorem 9]{Ananyevskiy_PhD_Thesis}). This implies that $ \varphi $ passes to the quotient, that is, we get a map:
	\[ \varphi: \quot{\mr A^{\bullet}(S)[p_1,p_2,\ldots, p_n, e]}{J_{2n, 2k+1}} \longrightarrow \mr A^{\bullet}( SGr_S(2n,2k+1))  \]
	
	By \cite[Theorem 4.2.9]{Asok-Deglise-Nagel}, we have Leray spectral sequences both for source and target of $ \varphi $:
	\[ E_1^{p,q}=\bigoplus_{s\in S^{(p)}} \quot{\mr A^{q}(\kappa(s))[p_1,p_2,\ldots, p_n, e]}{J_{2n, 2k+1}}  \allora \quot{\mr A^{p+q}(S)[p_1,p_2,\ldots, p_n, e]}{J_{2n, 2k+1}}  \]
	\[ 'E_1^{p,q}=\bigoplus_{s \in S^{(p)}} \mr A^{q}(SGr_{\kappa(s)}(2n,2k+1)) \allora \mr A^{p+q}(SGr_S(2n,2k+1)) \]
	The map between spectral sequences by $ \varphi $ is an isomorphism by \cite[Theorem 9]{Ananyevskiy_PhD_Thesis}, hence $ \varphi $ is an isomorphism too as claimed.
	
\end{proof}

\begin{thm}\label{ch3:_Ana_thm_9_NisLoc}
	Let $ S \in \catname{Sm}_{\bigslant{}{\sk}} $ and let $ \mc X \in \oocatname{ASt}_{\bigslant{}{S}}^{NL} $ be a smooth NL-algebraic stack. Let $ \mr A \in \mr{SH}(S) $ be an $ SL_{\eta} $-oriented ring spectrum. Then we have a map:
	\[ \varphi_{\mc X}: \quot{\mr A^{\bullet}(\mc X)[p_1,p_2,\ldots, p_n, e]}{J_{2n, 2k+1}} \longrightarrow \mr A^{\bullet}( \mc X\times_S SGr_{S}(2n,2k+1))  \]
	\noindent that is an isomorphism, where:
	\[ J_{2n,2k+1}:=\left( e^2-p_n, g_{k-n+1}(p_1,\ldots, p_n), g_{k-n+2}(p_1,\ldots,p_n), \ldots, g_k(p_1,\ldots,p_n) \right) \]
	\noindent and $ \varphi $ is defined by sending $ \varphi(p_i):=p_i(T_1) $ and $ \varphi(e):=e(T_1) $ with $ T_1 $ the tautological bundle of $ SGr_S(2n,2k+1) $.
\end{thm}
\begin{proof}
	Via the pullback map induced by the first projection $ p_1: \mc X\times_S SGr_S(2n,2k+1) $, we get an $ \mr A^{\bullet}(\mc X) $-algebra structure on $ \mr A^{\bullet}(\mc X \times_S SGr_S(2n,2k+1)) $. Since the Pontryagin and Euler classes in $ \mr A^{\bullet}(\mc X\times_S SGr_S(2n,2k+1)) $ are the pullback of the respective classes in $ \mr A^{\bullet}(SGr_S(2n,2k+1)) $, then there exists a unique map of $ \mr A^{\bullet}(\mc X) $-algebras:
	\[ \varphi_{\mc X}: \quot{\mr A^{\bullet}(\mc X)[p_1,p_2,\ldots, p_n, e]}{J_{2n, 2k+1}} \longrightarrow \mr A^{\bullet}( \mc X\times_S SGr_{S}(2n,2k+1))  \]
	\noindent defined by sending $ p_i $ to $  p_i(T_1) $ and $ e $ to $ e(T_1) $. Let $ x: X\rightarrow \mc X $ be a NL-atlas and let $ X^r_{\mc X}:=\check{C}_{r}(\bigslant{X}{\mc X}) $ be the scheme at the $ r^{th} $-level of the \v Cech nerve. By \cref{ch3:_Ana_Thm_9+} we have isomorphisms:
	\[ \varphi_{X^r_{\mc X}}: \quot{\mr A^{p,q}(X^r_{\mc X})[p_1,\ldots,p_n,e]}{J_{2n,2k+1}} \stackrel{\sim}{\longrightarrow} \mr A^{p+q}(X^r_{\mc X}\times_S SGr_S(2n,2k+1)) \]
	\noindent for each $ r $ and for each bi-degree $ (p,q) $. This implies that we have an equivalence of mapping spectra:
	
	\begin{equation}\label{ch3:_eq_Ana_9_NisLoc}
		\iMap(X^r_{\mc X}, A_{J_{2n,2k+1}}) \stackrel{\sim}{\longrightarrow} \iMap(SGr_{X^r_{\mc X}}(2n,2k+1), \mr A) 
	\end{equation}
	\[ \]
	\noindent where $ A_{J_{2n,2k+1}} $ is the ring spectrum representing $ \quot{\mr A^{\bullet}(-)[p_1,\ldots,p_n,e]}{J_{2n,2k+1}} $. The equivalence \eqref{ch3:_eq_Ana_9_NisLoc}, by \cite[Remark 3.38]{Motivic_Vistoli}, implies that we have:
	
	\begin{align*}
		\mr A_{J_{2n,2k+1}}(\mc X)&\simeq \lim_{r \in \Delta} \mr A_{J_{2n,2k+1}}(X^r_{\mc X})\simeq\\
		& \simeq \lim_{r \in \Delta} \iMap(X^r_{\mc X}, A_{J_{2n,2k+1}}) \simeq\\
		&\simeq \lim_{r \in \Delta} \iMap(SGr_{X^r_{\mc X}}(2n,2k+1), \mr A) \simeq\\
		&\simeq \lim_{r \in \Delta} \mr A(SGr_{X^r_{\mc X}}(2n,2k+1))\simeq \\
		& \simeq \mr A(\mc X\times_S SGr_S(2n,2k+1))
	\end{align*}
	In other words:
	\[  \varphi_{\mc X}: \quot{\mr A^{\bullet}(\mc X)[p_1,p_2,\ldots, p_n, e]}{J_{2n, 2k+1}} \longrightarrow \mr A^{\bullet}( \mc X\times_S SGr_{S}(2n,2k+1))  \]
	\noindent is an equivalence as claimed.
\end{proof}

\begin{thm}[K\"unneth for BSL]\label{ch3:_Kunneth_BSL_n}
	Let $ S \in \catname{Sm}_{\bigslant{}{\sk}} $ and let $ \mc X \in \oocatname{ASt}_{\bigslant{}{S}}^{NL} $ be a smooth NL-algebraic stack.
	Denote by $ \oocatname{U}_r $ the universal tautological rank $ r $ bundle over $ \mc BSL_{r,S} $. 
	Let $ \mr A \in \mr{SH}(S) $ be an $ SL_{\eta} $-oriented ring spectrum. Then there are unique $ \mr A^{\bullet}(\mc X) $-algebra maps: 
	\[ \varphi_{\mc X}: {\mr A^{\bullet}(\mc X)\llbracket p_1,p_2,\ldots, p_{n-1}, e\rrbracket} \longrightarrow \mr A^{\bullet}( \mc X\times_S \mc BSL_{2n,S})  \]
	\[ \varphi_{\mc X}: {\mr A^{\bullet}(\mc X)\llbracket p_1,p_2,\ldots, p_n\rrbracket} \longrightarrow \mr A^{\bullet}( \mc X\times_S \mc BSL_{2n+1,S})  \]
	\noindent that are continuous with respect to the topology given by the restriction to the finite level approximations $ \mc X\times_S SGr_S(2n,2k+1) $, resp. $ \mc X \times_S SGr_S(2n+1,2k+1) $, and with $ \varphi_{\mc X}(p_i)=p_i(\oocatname{U}_{2n}) $, resp. $ \varphi_{\mc X}(p_i)=p_i(\oocatname{U}_{2n+1})  $ and $ \varphi_{\mc X}(e)=e(\oocatname{U}_{2n})  $.
	
	\noindent Moreover the maps $ \varphi_{\mc X} $ are isomorphisms. In particular we get an isomorphism:
	\[ \mr A^{\bullet}(\prod_{j=0}^{s}\mc BSL_{k_j,S})\simeq \mr A^{\bullet}(\mc BSL_{k_0,S})\widehat{\otimes}_{\mr A^{\bullet}(S)} \ldots \widehat{\otimes}_{\mr A^{\bullet}(S)} \mr A^{\bullet}(\mc BSL_{k_s,S}) = \widehat{\bigotimes_j}\mr A^{\bullet}(\mc BSL_{k_j,S}) \]
\end{thm}
\begin{proof}
	Let us start with the even case (the only case we are actually interested in for future applications). We know that the ind-scheme $ BSL_{n,S}=\colim{k} SGr_{S}(n, 2k+1) $, since the system made of $ SGr_{S}(n,2k+1) $ is cofinal inside the system made by all special linear Grassmannians. Then \cref{ch3:_Ana_thm_9_NisLoc} tells us that the system $ \set{\mr A^{\bullet}(\mc X \times_S SGr_{S}(n,2k+1))} $ satisfies the Mittag-Leffler condition. Hence by \cite[Proposition 3.33]{Motivic_Vistoli}, applied to the spectrum $ \mr A_{\mc X} $ as constructed in \cite[Lemma 3.17]{Witt_Loc_PhD}, we have:
	\[ A^{\bullet}( \mc X\times_S \mc BSL_{n,S}) \simeq \lim_{k} \mr A^{\bullet}(\mc X \times SGr_{S}(n, 2k+1)) \]
	\noindent By \cref{ch3:_Ana_thm_9_NisLoc} this gives us the isomorphism:
	\[ \varphi_{\mc X}: {\mr A^{\bullet}(\mc X)\llbracket p_1,p_2,\ldots, p_{n-1}, e\rrbracket} \longrightarrow \mr A^{\bullet}( \mc X\times_S \mc BSL_{2n,S})  \]
	For the odd case, it is enough to use the identification $ \mr A^{\bullet}(SGr_S(2n+1, 2k+1))\simeq \mr A^{\bullet}(SGr_S(2k-2n,2k+1)) $ as pointed out in \cite[Remark 14]{Ananyevskiy_PhD_Thesis} and hence reduce to the same argument used in the even case.\\
	Now for the last statement of the theorem, it is enough to show that:
	\[  \mr A^{\bullet}(\mc BSL_{k_1,S}\times \mc BSL_{k_2,S})\simeq \mr A^{\bullet}(\mc BSL_{k_1,S})\widehat{\otimes}_{\mr A^{\bullet}(S)} \mr A^{\bullet}(\mc BSL_{k_2,S})  \] 
	\noindent and then iterate. Without loss of generality, we can suppose $ k_1=2n $ and $ k_2=2m+1 $. Let us denote the Pontryagin and Euler classes of $ \mc BSL_{k_1,S} $ by $ p_{i,1}, e_{1} $ and denote by $ p_{i,2} $ the Pontryagin classes of  $ \mc BSL_{k_2,S} $. Then by what we just proved, for $ \mc X= \mc BSL_{k_1,S} $, we have:
	\begin{align*}
		\mr A^{\bullet}(\mc BSL_{k_1,S}\times \mc BSL_{k_2,S})&\simeq \mr A^{\bullet}(\mc BSL_{k_1,S})\llbracket p_{1,2},\ldots,p_{m,2}\rrbracket \simeq \\
		& \simeq \left( \mr A^{\bullet}(S)\llbracket p_{1,1}, \ldots, p_{n-1,1},e_1\rrbracket \right)\llbracket p_{1,2},\ldots,p_{m,2}\rrbracket  \simeq \\
		&\simeq  \mr A^{\bullet}(S)\llbracket p_{1,1}, \ldots, p_{n-1,1},e_1\rrbracket \widehat{\otimes}_{\mr A^{\bullet}(S)} ( \mr A^{\bullet}(S)\llbracket p_{1,2}, \ldots, p_{m,2}\rrbracket \simeq \\
		& \simeq \mr A^{\bullet}(\mc BSL_{k_1,S}) \widehat{\otimes}_{\mr A^{\bullet}(S)} \mr A^{\bullet}(\mc BSL_{k_2,S})
	\end{align*}
	\noindent as we wanted. All the other cases are totally analogous and then we can iterate the computations for the general case $ \prod_{j=0}^{s}\mc BSL_{k_j}  $.
\end{proof}

\begin{co}[K\"unneth for BGL]\label{ch3:_Kunneth_BGL}
	Let $ S \in \catname{Sm}_{\bigslant{}{\sk}} $ and let $ \mc X \in \oocatname{ASt}_{\bigslant{}{S}}^{NL} $ be a smooth NL-algebraic stack. Let $ \oocatname{U}_r $ be the universal tautological bundle over $ \mc BGL_{r,S} $. Let $ \mr A \in \mr{SH}(S) $ be an $ SL_{\eta} $-oriented ring spectrum. Then there is a unique $ \mr A^{\bullet}(\mc X) $-algebra map:
	\[ \varphi_{\mc X}: {\mr A^{\bullet}(\mc X)\llbracket p_1,p_2,\ldots, p_{n}\rrbracket} \longrightarrow \mr A^{\bullet}( \mc X\times_S \mc BGL_{2n,S})\simeq \mr A^{\bullet}( \mc X\times_S \mc BGL_{2n+1,S})  \]
	\noindent that is continuous with respect to the topology given by the finite level approximations $ Gr_S(2n,k) $, resp. $ Gr_S(2n+1,k) $, and such that the elements $ p_i $ are sent to the Pontryagin classes $ p_i(\oocatname{U}_{2n}) $, resp. $ p_i(\oocatname{U}_{2n+1}) $.\\
	
	Moreover $ \varphi_{\mc X} $ is an isomorphism. In particular, we get an equivalence:
	\[ \mr A^{\bullet}(\prod_{j=0}^{s}\mc BGL_{k_j,S})\simeq \mr A^{\bullet}(\mc BGL_{k_0,S})\widehat{\otimes}_{\mr A^{\bullet}(S)} \ldots \widehat{\otimes}_{\mr A^{\bullet}(S)} \mr A^{\bullet}(\mc BGL_{k_s,S}) = \widehat{\bigotimes_j}\mr A^{\bullet}(\mc BGL_{k_j,S}) \]
\end{co}
\begin{proof}
	Under the identifications $ BGL_{2n,S}\simeq BGL_{2n+1,S}\simeq BSL_{2r+1} $ of \cite[Theorem 6.3.3, Theorem 6.3.7]{Haution_Odd_VB} in $ \mr{SH}(S)[\eta^{-1}] $, the corollary follows from \cref{ch3:_Kunneth_BSL_n}.
\end{proof}

\begin{pr}\label{ch3:_Kunneth_Formula} 
	Let $ S \in \catname{Sm}_{\bigslant{}{\sk}} $. Let $ \mr{A} $ be an $ SL_{\eta} $-oriented ring spectrum. Let $ \mc X=\prod_{i=1}^{s} \mc{B}GL_{n_i,S} \times \prod_{j=s+1}^{s+r} \mc{B}SL_{n_j,S}  $. Then we have:
	
	\[ \mr{A}^{\bullet}(X; L)\simeq \cbigotimes_{i=1}^{s} \mr{A}^{\bullet}\left( \mc{B}GL_{n_i}; L_i \right) \widehat\otimes_{W(\sk)} \cbigotimes_{j=s+1}^{s+r} \mr{A}^{\bullet}\left( \mc{B}SL_{n_j}; L_j \right) \]
	\noindent with $ L:=L_1 \boxtimes \ldots \boxtimes L_{s+r} $ and $ L_i\in \mr{Pic}(\mc{B}GL_{n_i}) $ for $ i=1,\ldots,s $ and $ L_j\in \mr{Pic}(\mc{B}SL_{n_j}) $ for $ j=s+1,\ldots,s+r $ and where $ \widehat{\otimes} $ denotes the completed tensor product.
\end{pr}
\begin{proof}
	The proposition follows by a Mayer-Vietoris argument from the untwisted cases in  \cref{ch3:_Kunneth_BSL_n} and \cref{ch3:_Kunneth_BGL}.
\end{proof}

\section{Twisted Symplectic Bundles and Twisted Borel Classes}

To aid in computing Euler classes of certain  interesting rank 2 bundles on $ \mc BN $, we introduce the notion of twisted Borel classes. We will always work over some base scheme $ S\in \catname{Sm}_{\bigslant{}{\sk}} $.  \\



In analogy with the fact that any rank two $ SL $-bundle has a canonical symplectic structure, we would like to consider a general rank two vector bundle with non-trivial determinant as a \textit{twisted symplectic bundle}:

\begin{defn}
	Given a vector bundle $ V $ over a scheme $ X $ and a line bundle $ L \in \mr{Pic}(X) $, an $ L $\textit{-twisted symplectic form} on $ V $ is a non-degenerate alternating form $ \omega^{L}: \Lambda^2 V \stackrel{}{\longrightarrow} L $.  A vector bundle $ V $ equipped with an $ L $-twisted symplectic form will be an $ L $-twisted symplectic bundle $ (V, \omega^{L}) $.
\end{defn} 

\begin{rmk}
	\begin{enumerate}\label{ch3:_rk_2_twisted_symplectic}
		\item Any rank 2 vector bundle $ V $ over any scheme $ X $ has a canonical $ \det(V) $-twisted symplectic form $ \omega_{can}: \Lambda^2 V \stackrel{\sim}{\longrightarrow} \det(V) $ given by the identification of its second exterior power with its determinant.
		\item For trivial twists $ L\simeq \mo{X} $, we recover the notion of symplectic bundles.
	\end{enumerate}
\end{rmk}

\begin{defn}
	A \textit{twisted symplectic Thom structure} on an $ SL $-oriented ring cohomology theory $ \mb E $ is a rule that assigns to each rank two $ L $-twisted symplectic bundle $ (E, \omega_L) $  over $ X $, a class:
	\[  \mr{th}(E, \omega^L) \in \mb E^{4,2}\left( \mr{Th}_X(E); L \right):=\mb E_{\mr{DJK}}^{6,3}\left( X, -[L\oplus E] \right)  \]
	\noindent such that:
	
	\begin{enumerate}
		\item For an isomorphism of twisted symplectic bundles $ u: (E_1,\omega_{L_1}) \stackrel{\sim}{\rightarrow} (E_2,\omega_{L_2}) $, we have  $ u^*\mr{th}(E_1, \omega^{L_1})=\mr{th}(E_2,\omega^{L_2}) $.
		\item For a map $ f: X \longrightarrow Y $ and a twisted symplectic bundle $ (E, \omega^L) $ over $ Y $, we have $ f^*\mr{th}(E, \omega^L)=\mr{th}(f^*E, f^*\omega^L) $.
		\item For $ can: \Lambda^2 \mo{X}^2 \stackrel{\sim}{\rightarrow} \mo{X} $ the canonical isomorphism, the class $ \mr{th}(\mo{X}^2, can) $ is the image of $ 1\in \mb E^{0,0}(X) $ under the suspension isomorphism:
		\[ \mb E^{0,0}(X)\simeq \mb E^{4,2}(\Sigma^{4,2}\mbbm 1_X)=\mb E^{4,2}(\Th{X}{\mo{X}^2}) \]
	\end{enumerate}
	
\end{defn}

\begin{notation}
	To distinguish between Thom classes coming from (twisted) symplectic bundles and the ones coming from $ SL $-vector bundle, we will denote the latter as $ th^{SL}(\cdot) $.
\end{notation}

The following is basically due to Ananyevskiy \cite[Corollary 1]{Ananyevskiy_Pushforwards_Eta_Inverted}:

\begin{pr}\label{ch3:_SL_oriented_Twisted_sympl}
	Any $ SL $-oriented ring spectrum $ \mb E $ admits a twisted symplectic Thom structure.
\end{pr}
\begin{proof}
	For a rank 2 vector bundle $ V $ over some $ X \in \catname{Sch}_{\bigslant{}{S}} $, we get the canonical identification $ \omega_{can}: \Lambda^2V \longrightarrow \det(V) $. Then we define:
	\[ \mr{th}(V, can):=\mr{th}^{SL}(V) \in \mb E^{4,2}(\Th{X}{V}; \det{}^{-1}(V)) \]
	\noindent where $ \mr{th}^{SL}(V)  $ is the $ SL $-Thom class living in the twisted cohomology as defined in \cref{ch2:_def_twisted_Thom_class}.

	
	It is not difficult to check that this class satisfies the desired relations, see for example \cite[\S 3]{Levine_Raksit_Gauss_Bonnet} for a detailed account of these classes (in the case $ X \in \catname{Sm}_{\bigslant{}{S}} $).
\end{proof}

\begin{defn}
	Given a scheme $ X $, $ L \in \mr{Pic}(X) $ and an $ L $-twisted symplectic bundle $ (V, \omega^L) $ of rank $ 2n+2 $, we define the $ L $-twisted quaternionic Grassmannian as:
	\[ \mr{HGr}^{L}_{X}(k, V ) :=\set{ W \in \mr{Gr}(k, V) \st{-3}{5} \ \restrict{\omega^L}{W}: \Lambda^2 W \rightarrow L \ \ \text{is non-degenerate} }\]
	\noindent We will often denote the corresponding tautological rank $ k $ bundle as $ \mc U^{L}_{V} $ if not specified otherwise. When $ k=2 $, we will call this the twisted quaternionic projective space and we will denote it with $ \mr H\pro^L_X(V) $.\\
	Let $ (V, \omega^L) $ be a twisted symplectic bundle of dimension $ 2n+2 $ over a scheme $ X $, then if the base scheme $ X $ is clear from the context we will just write:
	\[ \mr H\pro^n_L:=\mr H\pro^L_X(V) \]
\end{defn}

Consider now any rank 2 vector bundle $ V $ over $ X $ with its zero section $ s_0:X \rightarrow V $, $ \mb E $ an $ SL $-oriented ring spectrum. By the twisted Thom isomorphism (cf. \cref{ch2:_Twisted_Thom_SL}), we get an induced map $ s_*:  \mb E^{0,0}(X) \rightarrow \mb E^{4,2}(\mr{Th}_X(V); \det^{-1}(V)) $ such that $ s_*(\mbbm 1)=\mr{th}(V, \omega_{can})$. Post-composing with the map forgetting supports, we also get $ \bar{s}_*: \mb E^{0,0}(X) \rightarrow \mb E^{4,2}(V; \det^{-1}(V))  $. In particular the element $ \bar{s}_*(\mbbm 1) $ will be the image of the class $ (\mr{th}(V), \omega_{can}) $ through the map forgetting supports.

\begin{defn}
	Let $\mb E$ be an $ SL $-oriented ring spectrum with a twisted symplectic Thom structure, and let $ (V, \omega_L) $ be an $L$-twisted symplectic, rank 2, bundle over $ X $ with zero section $ s: X \rightarrow V $. Let $\alpha^*: \mb E^{\bullet,\bullet}(\Th{X}{V}; L)\rightarrow \mb E^{\bullet,\bullet}(V; L) $ be the map forgetting supports. We define the \textit{twisted Borel class} of $ V $ as:
	\[ b^{\mb E}_{L}(V, \omega_{L}):=-s^*\alpha^* \mr{th}(V,\omega_L) \in \mb E^{4,2}(X; L) \]
	\noindent If the ring spectrum $ \mb E $ is clear from the context we will drop it from the notation.
\end{defn}
\begin{rmk}
	If $ V $ is an $ SL $-bundle of rank 2 over $ X $, then its twisted Borel class $ b_{\mo{X}}(V, \omega_{can}) \in \mb E^{4,2}(\mr{Th}_X(V); \mo{X})=\mb E^{6,3}(\mr{Th}_X(V)\otimes \mr{Th}_X(\A^1_{X}))\simeq \mb E^{4,2}(\mr{Th}_X(V)) $ corresponds to the classical untwisted Borel class coming from the $ SL $-orientation of $ \mb E $, so no harm is done if we just refer to Borel class of $ V $.
\end{rmk}

\begin{thm}[Twisted Quaternionic Projective Bundle Theorem]\label{ch3:_Twisted_Proj_Bun_Thm}
	Let $ \mb E $ be an $ SL $-oriented ring spectrum with a twisted Thom structure. Let $ (V, \omega^L) $ be a twisted symplectic bundle of rank 2n over a scheme $ X\in \catname{Sm}_{\bigslant{}{S}} $, let $ (\mc U, \restrict{\omega^L}{\mc U}) $ be the tautological rank 2 bundle over $ \mr H\pro^n_L $ and let $ \zeta:=b(\mc U, \restrict{\omega^L}{\mc U}) $ be its Borel class. Write $ \pi: \mr H\pro^n_L \rightarrow X $ for the projection map. Then for any closed subset $ Z \sseq X $ we have the isomorphism  of $ \mb E(X) $-modules:
	\[ (1,\zeta,\zeta^2,\ldots,\zeta^{n-1}): \bigoplus_{j=0}^{n-1}\mb E_Z(X; L^{\otimes -j}) \longrightarrow \mb E_{\pi^{-1}(Z)}(\mr H\pro^n_L) \]
	Moreover there are unique classes $ b_j^{L}(\mc E, \omega^L) \in \mb E^{4j,2j}(X; L^{\otimes j}) $ such that:
	\[ \zeta^n-b_1(E, \omega^L) \cup \zeta^{n-1}+\ldots b_2(E, \omega^L) \cup \zeta^{n-2}-\ldots+(-1)^{n}b_n(E, \omega^L)=0 \]
	\noindent and if $ (E, \omega^L) $ is the trivial symplectic bundle then $ b_j(E, \omega^L)=0 $ for $ j=1, \ldots $.

\end{thm}

\begin{proof}
	The map we are looking for is given by:
	
	\[ \begin{array}{cccc}
		(1,\zeta,\zeta^2,\ldots,\zeta^{n-1}): &\bigoplus_{j=0}^{n-1}\mb E_Z(X; L^{\otimes -j}) & \longrightarrow & \mb E_{\pi^{-1}(Z)}(\mr H\pro^n_L)\\
		& (a_0,\ldots, a_{n-1}) & \mapsto & \sum_{0}^{n-1}  a_j\zeta^{j}
	\end{array} \]
	
	\noindent Find opens that trivialise $ L $ over $ X $ and use them with a Mayer-Vietoris argument to reduce to the untwisted case proved in \cite[Theorem 8.2]{Panin-Walter}.
\end{proof}

\begin{defn}
	The classes $ b_i^{L}(\mc E, \omega^L)  $ of the theorem above are the Borel classes associated to $ (V, \omega^L) $ with respect to the twisted Thom structure defined over $ \mb E $. For $ i>n $ and $ i<0 $ set $ b_i^{L}(\mc E, \omega^L)=0  $ and set $ b_0^{L}(\mc E, \omega^L)=1  $. We define the \textit{total twisted Borel class} as:
	\[ b_t^{L}(V, \omega^{L}):=1+b_1^{L}(V,\omega^{L})t+ \ldots + b_n^{L}(V,\omega^{L})t^n \]
\end{defn}

\begin{construction}
	Let  $ (E,\omega^{L}) $ be $ L $-twisted symplectic bundle of rank $ 2r $: we want now to construct a twisted version  of  the \textit{quaternionic flag bundle}. First consider:
	\[ \pi_1: \mr{HFlag}_X^{L}(1, r-1; E):=\mr{HGr}^{L}_X(2; E) \longrightarrow X \]
	Then over $  \mr{HFlag}_X^{L}(1, r-1; E) $ we have the tautological rank 2 twisted bundle $ U_{1,1} $ (the one classified by the identity map of $ \mr{HGr}^{L}_X(2; E) $). The bundle $ U_{1,1} $ is a sub-bundle of $ \pi_1^*E $, therefore we can consider:
	\[ U_{1,1}^{\perp}:=\set{w\in E\st{-1}{4} \ \omega^{L}(w,u)=0 \ \forall \ u \in U_{1,1}} \]
	\noindent that is the tautological sub-bundle of rank $ 2r-2 $ over $ \mr{HFlag}_X^{L}(1, r-1; E)  $. Let us rename $ U_{1,1}^{\perp} $ as:
	\[ W_{1}:=U_{1,1}^{\perp} \]
	This gives us a decomposition:
	\[(\pi_1^*E, \pi_1^*\omega^{L})\simeq \left(U_{1,1}, \omega_{1,1}\right) \perp \left(W_1, \omega_{W_1}\right)   \]
	\noindent where $ \omega_{1,1}:=\restrict{ \pi_1^*\omega^{L}}{U_{1,1}} $ and $ \omega_{W_1}:=\restrict{ \pi_1^*\omega^{L}}{W_1} $.\\
	
	Now we iterate; to ease the notation we will drop the upper and lower scripts from the twisted hyperbolic Grassmannians $ \mr{HGr} $. Consider the natural map:
	\[ p_2: \mr{HFlag}_X^L(1^2,r-2; E):= \mr{HGr}(2; W_1) \longrightarrow \mr{HFlag}_X^{L}(1, r-1; E) \]
	and denote by $ U_{2,2} $ the universal tautological rank 2 bundle over $ \mr{HFlag}_X^L(1^2,r-2; E) $. This bundle is a natural sub-bundle of $ p_2^*W_1 $ and again, naming $ W_2:=U_{2,2}^{\perp} $, we get:
	\[ (p_2^*W_1, p_2^*\omega_{W_1})\simeq  \left(U_{2,2}, \omega_{2,2}\right) \perp \left(W_2, \omega_{W_2}\right)   \]
	\noindent where $ \omega_{2, 2} $ and $ \omega_{W_2} $ are the restrictions of $ p_2^*\omega_{W_1} $ to $ U_{2,2} $ and $ W_2 $ respectively. In particular, if $ \pi_2: \mr{HFlag}_X^L(1^2,r-2; E)\rightarrow X $ is the map given by $ \pi_1\circ p_2 $, we have the following decomposition:
	\[ (\pi_2^*E, \pi_2^*\omega^{L})\simeq p_2^*\left(U_{1,1}, \omega_{1,1}\right) \perp \left(U_{1,2}, \omega_{1,1}\right) \perp \left(U_{2,2}, \omega_{2,1}\right)  \]
	Notice that on $ \mr{HFlag}_X^L(1^2,r-2; E) $ we have three tautological bundles, namely $ U_{1,2}:=p_2^*U_{1,1} $, $ U_{2,2} $ and $ W_2 $.\\
	Proceeding again in the same way, we can inductively construct:
	\[ p_k: \mr{HFlag}_X^L(1^k,r-k; E):=\mr{HGr}(2; W_{k-1})\longrightarrow \mr{HFlag}_X^L(1^{k-1},r-k+1; E) \]
	\noindent together with a natural map:
	\[ \pi_k:  \mr{HFlag}_X^L(1^k,r-k; E) \longrightarrow X \]
	\noindent Over $ \mr{HFlag}_X^L(1^k,r-k; E) $ we have the following (inductively defined) tautological bundles:
	\[ (U_{i,k}, \omega_{i,k}):=p_k^*(U_{i,k-1},  \omega_{i,k-1})\ \ \ \ i=1,\ldots, k-1 \]
	\[ (U_{k,k}, \omega_{k,k}) \]
	\[ (W_k, \omega_{W_k}):=(U_{k,k}^{\perp}, \omega_{U_{k,k}^{\perp}} )\]
	\noindent where $ U_{k,k} $ is the tautological rank 2 bundle of $ \mr{HGr}(2; W_{k-1}) $.	 By construction, we have:
	\[ \pi_k^*(E,\omega^{L})\simeq \displaystyle\bigperp_{i=1}^{k} \left( U_{i,k}, \omega_{i,k} \right) \perp \left( W_k, \omega_{W_k} \right)  \]
\end{construction}	 
\begin{defn}
	Let $ (E, \omega^L) $ be an $ L $-twisted symplectic bundle of rank $ 2r $, then we define the \textit{complete twisted quaternionic flag bundle} as:
	\[ \mr{HFlag}_{X}^{L}(E,\omega^{L}):=\mr{HFlag}_{X}^{L}(1^{r}, 0; E)   \]
	\noindent and we denote its tautological rank 2 bundles as $ (\mc U_i, \omega_i^{L}):=(U_{1,r}, \omega_{i,r} )$ for $ i=1,\ldots, r $.
\end{defn}

Let $ q:  \mr{HFlag}_{X}^{L}(E,\omega^{L}) \rightarrow X $ be the projection map, then by construction we have:
\[ q^*(E, \omega^{L}) \simeq \bigoplus_{i=1}^{r} \left( \mc U_i, \omega^{L}_i \right) \]

\begin{defn}
	Let $ (E, \omega^L) $ be an $ L $-twisted symplectic bundle of rank $ 2r $ and let $ \mr{HFlag}_{X}^{L}(E,\omega^{L}) $ be its associated complete flag variety. Given $ \mc U_i $ tautological rank 2 bundle over $ \mr{HFlag}_{X}^{L}(E,\omega^{L})  $, we define the $ i^{th} $ \textit{twisted Borel root} of $ (E, \omega^{L}) $ to be:
	\[ u_i:=b\left( \mc U_i, \omega^{L}_i \right) \]
\end{defn}

\begin{defn}
	Let $ GL_{2,\det}^{\times r} $ be the following sub-group of $ GL_2^{\times r} $:
	\[ GL_{2,\det}^{\times r}:=\set{(g_1,\ldots,g_r)\in GL_2^{\times r}\st{-1}{4} \ \det(g_1)=\ldots=\det(g_r)}\sseq GL_2^{\times r} \]
\end{defn}
We have an inclusion $ SL_2^{\times r}\sseq GL_{2,\det}^{\times r} $ fitting into the following exact sequence:
\[ 1 \rightarrow SL_2^{\times r} \longrightarrow GL_{2,\det}^{\times r} \stackrel{\det(-)}{\longrightarrow} \mb G_m \rightarrow 1 \]
\begin{rmk}\label{ch3:_rmk_BGL_det}
	Let $ \oocatname{D}_r $ be the universal bundle over $ \mc BGL_{2,\det}^{\times r} $. The quotient stack $ \mc BGL_{2,\det}^{\times r} $ is the stack classifying $ r $-tuples of rank 2 vector bundles $ V_i $, over some scheme $ X $, together with isomorphisms $ \rho_i: \det(V_i)\stackrel{\sim}{\rightarrow} L $ for $ L \in \mr{Pic}(X) $. In other words, if  $ \oocatname{V}:=(\set{V_i}_i, L, \set{\rho_i}_i ) $ denotes a collection of such objects over $ X $, then there exists a classifying map $ f_{\oocatname{V}} $ fitting in the following cartesian diagram:
	\begin{center}
		\begin{tikzpicture}[baseline={(0,0)}, scale=1.75]
			\node (a) at (0,1) {$ \oocatname{V} $};
			\node (b) at (1, 1) {$ \oocatname{D}_r $};
			\node (c)  at (0,0) {$  \mc X $};
			\node (d) at (1,0) {$ \mc BGL_{2,\det}^{\times r}  $};
			\node (e) at (0.2,0.8) {$ \ulcorner  $};
			\node (f) at (0.5,0.5) {$  $};

			\path[font=\scriptsize,>= angle 90]
			
			(a) edge [->] node [above ] {$  $} (b)
			(a) edge [->] node [left] {$  $} (c)
			(b) edge[->] node [right] {$  $} (d)
			(c) edge [->] node [below] {$ f_{\oocatname{V}} $} (d);
		\end{tikzpicture}
	\end{center}
	The map:
	\[  GL_{2,\det}^{\times r} \stackrel{\det(-)}{\longrightarrow} \mb G_m \]
	\noindent induces a map:
	\[ \widetilde{\det}: \mc BGL_{2,\det}^{\times r} \stackrel{}{\longrightarrow} \mc B\mb G_m   \]
	Denote by $ \mc U_{\mc B\mb G_m} $ the universal bundle of $ \mc B\mb G_m $. The universal bundle $ \oocatname{D}_r $ is given by an $ r $-tuple of rank 2 bundles $ \mc V_i $, together with isomorphisms $ \det(\mc V_i)\simeq \widetilde{\det}^*\mc U_{\mc B\mb G_m} $. 
\end{rmk}
\vspace{1cm}


The inclusion $ SL_2^{\times r}\sseq GL_{2,\det}^{\times r}  $ gives rise to a natural map:
\begin{equation*}
	\mf i_0:  BSL_2^{\times r} \longrightarrow  BGL_{2,\det}^{\times r}
\end{equation*}
Let $\mo{}(1)$ be the tautological bundle of $ B\mb G_m $; with a little abuse of notation we will denote with $ \mo{}(1) $ also the pullback to $ BGL_{2,\det}^{\times r} $ of the tautological bundle of $ B\mb G_m $ along the natural map $ BGL_{2,\det}^{\times r} \stackrel{}{\longrightarrow} B\mb G_m $. The pullback $ \mf i_0^*\mo{}(1) $ is isomorphic to the determinant of the rank 2 tautological bundles of $ BSL_2^{\times r} $ and therefore it gets trivialised. Let $ \vartheta: \det(\mf i_0^*\mo{}(1))\stackrel{\sim}{\rightarrow} \mo{BSL_2^{\times r}} $ be such trivialization. For any $ SL $-oriented spectrum $ \mr A  $, via the (inverse of the) Thom isomorphism associated to $ \vartheta $, we have an identification map:
\[ \vartheta_*: \mr A^{\bullet, \bullet}(BSL_2^{\times r}; \mf i_0^*\mo{}(1)) \stackrel{\sim}{\longrightarrow} \mr A^{\bullet,	\bullet}(BSL_2^{\times r}) \]
Then the pullback $ \mf i_0^* $ induces the two following maps:
\[ \mf i_0^*: \mr A^{\bullet,\bullet}(BGL_{2,\det}^{\times r}) \longrightarrow \mr A^{\bullet,\bullet}(BSL_2^{\times r}) \]
\[ \mf i_1^*:=\vartheta_*\circ \mf i_0^*: \mr A^{\bullet,\bullet}(BGL_{2,\det}^{\times r}; \mo{}(1))\longrightarrow \mr A^{\bullet,\bullet}(BSL_2^{\times r}) \]
Considering $ \mf i_0^* $ and $ \mf i_1^* $ together we get the map:
\begin{equation}\label{ch3:_eq_GL_det_MAP}
	\mf i^*: \mr A^{\bullet,\bullet}(BGL_{2,\det}^{\times r})\oplus \mr A^{\bullet,\bullet}(BGL_{2,\det}^{\times r}; \mo{}(1)) \longrightarrow \mr A^{\bullet,\bullet}(BSL_{2}^{\times r})
\end{equation}

\begin{pr}\label{ch3:_MEC_4.1_BGL_DET}
	Let $ \mr A $ be an $ SL_{\eta} $-oriented ring spectrum. 
	Then the map \eqref{ch3:_eq_GL_det_MAP}:
	\[ \mf i^*: \mr A^{\bullet}(\mc BGL_{2,\det}^{\times r})\oplus \mr A^{\bullet}(\mc BGL_{2,\det}^{\times r}; \mo{}(1)) \longrightarrow  \mr A^{\bullet}(\mc BSL_{2}^{\times r}) \]
	\noindent is an isomorphism.
\end{pr}
\begin{proof}
	Recall that we have:
	\[ 1 \rightarrow SL_2^{\times r} \longrightarrow GL_{2,\det}^{\times r} \stackrel{\det(-)}{\longrightarrow} \mb G_m \rightarrow 1 \]
	The bundle $ \mo{}(1) $ is the tautological bundle of $ B\mb G_m $ pulled back to $ B GL_{2,\det}^{\times r}  $, and we have that $ BSL_2^{\times r}\simeq \mo{}(1)\setminus \set{0} $. Then, we can identify $ BSL_2^{\times r} \longrightarrow  BGL_{2,\det}^{\times r} $ with the natural map $ \mo{}(1)\setminus \set{0}\rightarrow BGL_{2,\det}^{\times r} $ given by the composition of the open immersion $ j: \mo{}(1)\setminus\set{0}\into \mo{}(1) $ with the projection map from $ \mo{}(1) $. From here, the same argument used to prove the first part of \cref{ch3:_4.1_MEC} applies verbatim replacing $ BSL_n $ with $ BSL_2^{\times r} $ and $ BGL_n $ with $ BGL_{2,\det}^{\times r} $. We will sketch the proof again for more clarity, but we will leave the details to the reader. \\
	By homotopy invariance we can identify $ \mr A(BGL_{2,\det}^{\times r})\simeq \mr A(\mo{}(1)) $, and we get a localization sequence:
	\begin{align}\label{ch3:_eq_BGL_det_Untwisted}
		\begin{split}
			\ldots \rightarrow \mr A^{a,b}\left( BGL_{2,\det}^{\times r}\right)  &\stackrel{}{\rightarrow} \mr A^{a,b}\left( {B}SL_2^{\times r} \right) -\!\!\!-\!\!\!-\ldots  \\
			\ldots&\stackrel{\partial_{a,b}}{\rightarrow} \mr A^{a-1,b-1}\left( BGL_{2,\det}^{\times r}; \mo{}(1) \right)  \stackrel{e(\mo{}(1))\cup \cdot}{\rightarrow} \mr A^{a+1,b}\left( BGL_{2,\det}^{\times r} \right)  \rightarrow \ldots
		\end{split}
	\end{align}
	By \cite[Lemma 4.3]{Motivic_Euler_Char} this sequence splits into short exact sequences:
	\[ 	0 \rightarrow \mr A^{a,b}\left( BGL_{2,\det}^{\times r}\right)  \stackrel{}{\rightarrow} \mr A^{a,b}\left( {B}SL_2^{\times r} \right) \stackrel{\partial_{a,b}}{\rightarrow} \mr A^{a-1,b-1}\left( BGL_{2,\det}^{\times r}; \mo{}(1)\right) \rightarrow 0  \]
	To conclude, we need to find a splitting for $ \partial_{a,b} $ and the candidate map is given by:
	\[ \mf i_1^*:= \vartheta_*\circ \mf i_0^*: \mr A^{\bullet}(BGL_{2,\det}^{\times r}; \mo{}(1))\longrightarrow \mr A^{\bullet}(BSL_2^{\times r})  \]
	We now want to prove that $ \partial\circ \mf i_1^* $ is an isomorphism. Let $ \bar\partial $ be the boundary map in the twisted localization sequence:
	\begin{equation*}
		\ldots \rightarrow \mr A^{a,b}\left( {B}GL_n; \mo{}(-1) \right)  \stackrel{}{\rightarrow} \mr A^{a,b}\left( {B}SL_n; \mf i_0^*\mo{}(-1)\right)  \stackrel{\bar\partial_{a,b}}{\rightarrow} \mr A^{a-1,b-1}\left( {B}GL_n\right)  \stackrel{}{\rightarrow} \ldots 
	\end{equation*}
	Let $ \langle t_{can}\rangle \in \mr A^{0}(BSL_2^{\times r}; \mf i_0^*\mo{}(1))\simeq \mr A^{0}(BSL_2^{\times r}; \mf i_0^*\mo{}(-1)) $ given by the tautological section of $ \mo{}(1) $. Then we have:
	\[ \partial\circ\mf i_1^*(-)=\bar{\partial}(\langle t_{can}\rangle) \cup - \]
	To check that multiplication by $ \bar{\partial}(\langle t_{can}\rangle)  $ is an isomorphism, by the Milnor's $ \lim^1 $-exact sequence, we can reduce to prove the claim on the finite level approximations $ B_mSL_2^{\times r} $ and $ B_mGL_{2,\det}^{\times r} $. By a Mayer-Vietoris argument, we can further restrict to opens of $ B_mGL_{2,\det}^{\times r} $ where $ \mo{}(1) $ gets trivialised. But then the restriction of $  \bar{\partial}(\langle t_{can}\rangle)  $ becomes just $ \eta $ by \cite[Lemma 6.4]{Ananyevskiy_SL_oriented}, and hence the map $ \bar{\partial}(\langle t_{can}\rangle) \cup -  $ is invertible as claimed.
\end{proof}

\begin{thm}[Twisted Symplectic Splitting Principle]
	Let $ \mb E $ be an $ SL_{\eta} $-oriented ring spectrum with its canonical twisted Thom structure and let $ X\in \catname{Sm}_{\bigslant{}{S}}  $ be a smooth $ S $-scheme. Take $ L\in \mr{Pic}(X) $, let $ (E,\omega^L)  $ be a twisted symplectic bundle and let $ q:\mr{HFlag}^L_X(E,\omega^L)\rightarrow X  $ be the associated complete twisted quaternionic flag bundle. Then the map:
	\[ q^*: \mb E^{\bullet}(X) \longrightarrow \mb E^{\bullet}\left( \mr{HFlag}^{L}_{X}(E,\omega^{L}) \right) \]
	\noindent is injective. Moreover we have that:
	\[ q^*b_t^{L}(E,\omega^{L})=\prod_{j=1}^{r} b_t(\mc U_i, \omega^{L}_i) \]
	\noindent with $ (\mc U_i,\omega^{L}_i) $ the universal rank $ 2 $ bundles of $ \mr{HFlag}_{X}^{L}(E, \omega^{L})  $.
\end{thm}
\begin{proof}
	The first claim is just an iterated application of the Twisted Projective Bundle \cref{ch3:_Twisted_Proj_Bun_Thm}. For the remaining claim, we already noticed that $ q^*E \simeq \bigoplus_{i=1}^{r} \left( \mc U_i, \omega^{L}_i \right)  $. Now these $ (\mc U_i, \omega^{L}_i) $ are twisted symplectic bundles of rank 2, but that means they are just rank 2 vector bundles together with isomorphisms $ \omega^{L}_i: \det\left(\mc U_i\right) \stackrel{\sim}{\rightarrow} \restrict{L}{\mc U_i} $. Therefore, we get a natural map:
	\[ p: \mr{HFlag}^{L}_X(E,\omega^{L}) \longrightarrow \mc{B}GL_{2,\det}^{\times r} \]
	\noindent classifying $ (\set{\mc U_i}, L, \set{\omega_i}) $ (cf. \cref{ch3:_rmk_BGL_det}).
	The universal bundle $ \oocatname{D}_r $ of $ \mc{B}GL_{2,\det}^{\times r}  $ is given by an $ r $-tuple of rank $ 2 $ bundles $ 	\mc V_i $, together with isomorphisms $ \rho_i: \det(\mc V_i)\stackrel{\sim}{\rightarrow}\mo{}(1) $, where $ \mo{}(1) $ is  the pullback of the tautological bundle over $ \mc B\mb G_m $. We can regard the $ \mc V_i $ as twisted symplectic bundles, with twisted symplectic forms given by $ \omega_{i,can}: \Lambda^2 \mc V_i\simeq \det(\mc V_i) \stackrel{\rho_i}{\rightarrow} \mo{}(1)  $. Then by construction of $ p $, we get that:
	\[ p^*(\mc V_i, \omega_{i,can})\simeq (\mc U_i, \omega_i^{L}) \]
	
	If we prove the claim of our theorem for the bundle $ \bigoplus (\mc V_i, \omega_{i,can}) $, then, via the pullback map $ p^* $, we get the relation we want for $ q^*(E,\omega^{L}) $.
	This means that it is enough to prove the following:
	\[ b_t^{\mo{}(1)}\left(\bigoplus (\mc V_i, \omega_{i,can})\right)=\prod b_t^{\mo{}(1)}((\mc V_i, \omega_{i,can})) \]
	By \cref{ch3:_MEC_4.1_BGL_DET}, we have an injective map:
	\begin{equation}\label{ch3:_eq_inj_BGL_det}
		\mf i_1^*:\mr A^{\bullet}\left( \mc{B}GL_{2,\det}^{\times r}; \mo{}(1)\right) \into \mr A^{\bullet}\left( \mc{B}SL_2^{\times r} \right)
	\end{equation}
	\noindent induced by $ \mf i_0: \mc{B}SL_2^{\times r} \rightarrow \mc{B}GL_{2,\det}^{\times r} $. For each $ i $, the pullback $ \mf i_0^* \mc V_i $ is isomorphic to the universal rank 2 $ SL $-vector bundle $ \mc W_i $ of $ \mc{B}SL_2^{\times r} $.  Recall that the map $ \mf i_1^* $ was given by the pullback map $ \mf i_0^* $, together with the inverse of the Thom isomorphism associated to the determinant of the $ SL $-vector bundles $ \mf j^*\mc V_i $. By injectivity of \eqref{ch3:_eq_inj_BGL_det}, we can therefore reduce ourselves to prove:
	\[ b_t\left( \bigoplus (\mc W_i, \psi_{i,can}) \right)=\prod (\mc W_i, \psi_{i,can}) \]
	\noindent with $ \psi_{i,can} $ the canonical symplectic form given by their $ SL $-structure. But this amounts to the usual Cartan Sum formula for (untwisted) symplectic bundles, and therefore we can conclude by \cite[Theorem 10.5]{Panin-Walter}.
\end{proof}

\begin{co}[Twisted Cartan Sum Formula]\label{ch3:_Cartan_Sum_Formula_Twisted}
	If we have $ \left( F, \psi^{L} \right)\simeq \left( E_1, \omega_1^{L} \right)\oplus (E_2,\omega_2^{L})  $ a direct sum of $ L $-twisted symplectic bundles, then:
	\[ b_t^{L}(F,\psi^{L})=b_{t}^{L}(E_1,\omega_1^{L})b_t^{L}(E_2,\omega_2^{L}) \]
	\[ b_i^{L}(F, \psi^{L})=\sum_{j=0}^{i} b_{i-j}^{L}(E_1,\omega_1^{L})b_{j}^{L}(E_2, \omega_2^{L}) \]
\end{co}
\begin{proof}
	Using the twisted symplectic splitting principle, we can just follow the same steps as in \cite[Theorem 10.5]{Panin-Walter}.
\end{proof}

Remember that the spectrum $ \mr{BO}^{\bullet,\bullet} $ defined in \cite{Panin-Walter} is $ (8,4) $-periodic, with periodicity isomorphism:
\[ \mr{BO}^{\bullet+8,\bullet+4} \stackrel{\cdot \cup \Sigma^{8,4}\gamma}{\longrightarrow } \mr{BO}^{\bullet,\bullet}    \]
\noindent with $ \gamma \in \mr{BO}^{-8,-4}(pt)  $ the element corresponding to $ \mbbm 1 \in \mr{BO}^{0,0}(pt)  $ under the periodicity isomorphism:
\[ \mr{BO}^{0,0}(pt)\simeq \mr{GW}^{[0]}_{0}(pt)\simeq \mr{GW}^{[-4]}_{0}(pt)\simeq\mr{BO}^{-8,-4}(pt)   \]
\noindent and $ \mr{GW}^{[n]}_{i} $ are the higher Grothendieck-Witt groups of  \cite{schlichting_Hermitian_2010}. We will call $ \gamma $ the \textit{Bott element of Hermitian K-Theory}. This element will induce the 4 periodicity on $ \mr{KW} $ once we invert the $ \eta $-map. \\

Now  to explicitly compute the Euler classes of $ \widetilde{\mo{}^{\pm}}(m) $, we need a twisted version of  \cite[Lemma 8.2]{Ananyevskiy_FTL_Witt}:

\begin{lemma}[Ananyevskiy]\label{ch3:_Ana_Stab_Op_Lemma}
	Let $ E_1,E_2,E_3 $ be rank 2 bundles over some scheme $ X\in \catname{Sm}_{\bigslant{}{S}}  $, with the determinants $ L_1,L_2,L_3 $ respectively, together with their canonical twisted symplectic structures; let $ E:=E_1\otimes E_2 \otimes E_3 $ be the $ L_1\otimes L_2\otimes L_3 $-twisted symplectic bundle of rank 8, with the induced twisted symplectic structure. Let $ L:=L_1\otimes L_2\otimes L_3  $, $ \xi_i:=b^{L}_1(E_i) \in \mr{KW}^{4,2}(X; L_i) $ and denote with $ \sigma(n_1,n_2,n_3) $ the sum of all the monomials lying in the orbit of $ \xi_1^{n_1}\xi_2^{n_2}\xi_3^{n_3} $ under the action of $ S_3 $. Then:
	\begin{align*}
		& b_1^{L}(E)=\gamma \sigma(1,1,1)\\
		& b_2^{L}(E)=\gamma \sigma(2,2,0)-2\sigma(2,0,0)\\
		& b_3^{L}(E)=\gamma \sigma(3,1,1) -8\sigma(1,1,1)\\
		& b_4^{L}(E)=\gamma \sigma(2,2,2)+\sigma(4,0,0)-2\sigma(2,2,0)
	\end{align*} 
\end{lemma}
\begin{rmk}
	Notice that $ (mod \ 2) $, the sequences $ (n_1,n_2,n_3) $ (and their permutations)  appearing in the formulas above for $ b_i^{L}(E) $ are the same. This means that the terms land in the same twists, under the canonical identifications:
	\[ \mr{KW}^{\bullet}(X; L_1^{a_1}\otimes L_2^{a_2}\otimes L_3^{a_3}) \simeq \mr{KW}^{\bullet}(X; L_1^{b_1}\otimes L_2^{b_2}\otimes L_3^{b_3})  \]
	\noindent for:
	\[  (a_1,a_2,a_3)\equiv (b_1,b_2,b_3) \ \ \ \ (mod \ 2)\]
	
\end{rmk}
\begin{proof}[Proof of \ref{ch3:_Ana_Stab_Op_Lemma}]
	It is enough to prove the theorem for $ X=\mc{B}GL_2^{\times 3} $, but again by the K\"unneth formula \ref{ch3:_Kunneth_Formula} and \cref{ch3:_4.1_MEC} we can reduce to $ \mc{B}SL_2^{\times 3} $ case. Taking the finite level approximation $ \mc{B}_mSL_2 $ we can use Ananyevskiy's result \cite[Lemma 8.2]{Ananyevskiy_FTL_Witt}  and since Witt theory for $ \mc{B}SL_2 $ satisfies the Mittag-Leffler condition by \cite[Theorem 9]{Ananyevskiy_PhD_Thesis} (or by \cref{ch3:_Ana_Thm_9+}), taking the limit of the theories on the finite level approximation we get the desired result for $ {B}SL_2 $.
\end{proof}

\section{$\mr{KW}$-Euler classes for $ \widetilde{\mo{}}^{\pm}(m) $}

Before actually computing the Euler classes we are interested in, let us recall some facts on about the rank 2 vector bundles of ${B}N$. Let us identify:

\[ \begin{array}{cccc}
	\iota: & \mb G_m & \into &  T_{SL_2}\\
	& t & \mapsto & \left( \begin{array}{cc}
		t & 0 \\
		0 & t^{-1}
	\end{array} \right)
\end{array} \]
\noindent where $ T_{SL_2} $ is the torus in $ SL_2 $. We will denote as done before:
\[ \sigma:=\left( \begin{array}{cc}
	0 & 1 \\
	-1 & 0
\end{array} \right) \in N \]

\begin{defn}
	For $ m\geq 1 $ an integer, let $ F^{\pm}(m):=\sk^2 $, then we define representations $ (F^{\pm}(m), \rho^{\pm}_{m}) $ by:
	
	\[ \rho^{\pm}_m\left(  \iota(t)  \right):= \left( \begin{array}{cc}
		t^m& 0 \\
		0 & t^{-m}
	\end{array} \right) \]
	\[ \rho^{\pm}_m\left( \sigma  \right):= \pm \left(\begin{array}{cc}
		0 & 1 \\
		(-1)^m & 0
	\end{array}\right)  \]
	\noindent We call $ (F_0, \rho_0) $ the trivial one-dimensional representation, and $ (F_0, \rho_0^{-}) $ the one defined by $ \rho_0^{-}(t)=1 $ and $ \rho_0^{-}(\sigma)=-1 $.\\
	\noindent	We denote with $ p^{(\pm m)}: \widetilde{\mo{}}^{\pm}(m) \longrightarrow {B}N $ the rank 2 vector bundle:
	\[ p^{(\pm m)}: F^{\pm}_m \times^N {E}SL_2 \longrightarrow \bigslant{{E}SL_2}{N}={B}N \]
	
	\noindent corresponding to $ \rho^{\pm}_m $, with zero section $ s^{(\pm m)}_0: {B}N \longrightarrow \widetilde{\mo{}}^{\pm}(m) $ and $ \mr{th}^{(\pm m)} \in \mr{KW}\left( \widetilde{\mo{}}^{\pm}(m), \det(\widetilde{\mo{}}^{\pm}(m)) \right) $ the Thom class defined as $ (s^{(\pm m)}_0){}_*(\mbbm 1) $.
	
\end{defn}

\begin{rmk}
	We recall that $ \mr{Pic}\left( {B}N \right)\simeq \bigslant{\Z}{2\Z} $ can be generated by $ \gamma:= \widetilde{\mo{}}^{-}(0) $. For the one dimensional representation $ \det(\rho_m^{\pm}) $ we choose the generator given by $ e_1\wedge e_2 $, with $ e_1=(1,0) $ and $ e_2=(0,1) $ the standard basis of $ \sk^2 $. Then we have canonical isomorphisms $ \det\left( \widetilde{\mo{}}^{\pm}(m) \right)\simeq \mo{{B}N} $ for $ m $ odd and $ \det\left( \widetilde{\mo{}}^{\pm}(m) \right)\simeq \gamma $ for $ m>0 $ even.
\end{rmk}

Given a triple tensor product of rank two bundles $ U_1 \otimes U_2 \otimes U_3 $, we denote their bases as $ \set{e_1,e_2}, \set{f_1,f_2}, \set{g_1,g_2} $.
We have the isomorphism:
\[ \widetilde{\mo{}}^{+}(1)\otimes \widetilde{\mo{}}^{+}(1)  \otimes \widetilde{\mo{}}^{+}(1) \simeq   \widetilde{\mo{}}^{+}(3)  \oplus  \widetilde{\mo{}}^{+}(1)^{\oplus 3} \]
\noindent where the base chosen for this identification is given by dual pairs $ \set{v_i,w_i} $ that are perpendicular to all other vectors:
\begin{align*}
	v_1:= \ \ \ \! e_1\otimes f_1 \otimes g_1 \ \hspace{1em} & w_1=e_2\otimes f_2 \otimes g_2 \\
	v_2:= -e_1\otimes f_1 \otimes g_2 \ \hspace{1em} & w_2=e_2\otimes f_2 \otimes g_1 \\
	v_3:= -e_1\otimes f_2 \otimes g_1 \ \hspace{1em} & w_3=e_2\otimes f_1 \otimes g_2\\
	v_4:= -e_2\otimes f_1 \otimes g_1 \ \hspace{1em} & w_4=e_1\otimes f_2 \otimes g_2
\end{align*}
In a similar way, for any $ m>1 $ we have:
\begin{equation}\label{ch3:_decomposition_Om_for_recursive formulas}
	\widetilde{\mo{}}^{+}(m)\otimes \widetilde{\mo{}}^{+}(1)  \otimes \widetilde{\mo{}}^{+}(1)\simeq \widetilde{\mo{}}^{+}(m+2)\oplus \widetilde{\mo{}}^{+}(m) ^{\oplus 2} \oplus \widetilde{\mo{}}^{+}(m-2)  
\end{equation}
\noindent where the base is given by:
\begin{align*}
	v_1:= \ \ \ \! e_1\otimes f_1 \otimes g_1 \ \hspace{1em} & w_1=e_2\otimes f_2 \otimes g_2 \\
	v_2:= -e_1\otimes f_1 \otimes g_2 \ \hspace{1em} & w_2=e_2\otimes f_2 \otimes g_1 \\
	v_3:= -e_1\otimes f_2 \otimes g_1 \ \hspace{1em} & w_3=e_2\otimes f_1 \otimes g_2\\
	v_4:= \ \ \ \! e_1\otimes f_2 \otimes g_2 \ \hspace{1em} & w_4=e_2\otimes f_1 \otimes g_1
\end{align*}

\noindent We also have:
\begin{equation}\label{ch3:_Decomposition_O(2)_squared}
	\widetilde{\mo{}}^{+}(2)\otimes \widetilde{\mo{}}^{+}(2)  \otimes \widetilde{\mo{}}^{+}(1)\simeq \widetilde{\mo{}}^{+}(5)\oplus \widetilde{\mo{}}^{+}(3) ^{\oplus 2} \oplus \widetilde{\mo{}}^{-}(1)  
\end{equation}

\noindent with with bases:

\begin{align*}
	v_1:= \ \ \ \! e_1\otimes f_1 \otimes g_1 \ \hspace{1em} & w_1=e_2\otimes f_2 \otimes g_2 \\
	v_2:= -e_1\otimes f_1 \otimes g_2 \ \hspace{1em} & w_2=e_2\otimes f_2 \otimes g_1 \\
	v_3:= \ \ \ \! e_1\otimes f_2 \otimes g_1 \ \hspace{1em} & w_3=e_2\otimes f_1 \otimes g_2\\
	v_4:= \ \ \ \! e_2\otimes f_1 \otimes g_1 \ \hspace{1em} & w_4=e_1\otimes f_2 \otimes g_2
\end{align*}

\begin{pr}\label{ch3:_recursive_formulas_proposition}
	In $\mr{KW}\left(  BN, \det(\widetilde{\mo{}}^{\pm}(m)) \right) $, we have the following:
	\begin{enumerate}
		\item For any $ m $:
		\[ e(\widetilde{\mo{}}^{-}(m))=-e(\widetilde{\mo{}}^{+}(m))  \]
		\item For $ m>1 $ the recurrence relation:
		\[ b_1(\widetilde{\mo{}}^{+}(m+2))=(\gamma e^2-2)b_1(\widetilde{\mo{}}^{+}(m))-b_1(\widetilde{\mo{}}^{+}(m-2)) \]
		\noindent In particular $ b_1(\widetilde{\mo{}}^{+}(m)) $ is always a multiple of $ e:=e(\widetilde{\mo{}}^{+}(1) ) $ for $ m $ odd and $ \tilde{e}:=e(\widetilde{\mo{}}^{+}(2))  $ for $ m $ even respectively.
		\item For $ m=2n+1 $:
		\begin{equation} \label{ch3:_Odd_recursive_formula}
			b_1\left(\widetilde{\mo{}}^{+}(m)\right)=\sum_{k=0}^{n} (-1)^{n-k}\alpha_{k,n}\gamma^{k}e^{2k+1}
		\end{equation}
		\noindent with $ e=e(\widetilde{\mo{}}^{+}(1) ) $ and $ \alpha_{k,n} $ defined by the recurrence relation:
		\[ \begin{cases*}
			\alpha_{0,n}=2n+1 & $ \forall \ n\geq 0 $ \\
			\alpha_{k,n}=\sum_{j=1}^{n} j \cdot \alpha_{k-1, n-j} & $\forall \ n \geq  k>0 $\\
			\alpha_{k,n}=0 & else 
		\end{cases*} \]
		\item For $ m=2n $:
		\begin{equation} \label{ch3:_Even_recursive_formula}
			b_1\left(\widetilde{\mo{}}^{+}(m)\right)=\tilde{e}\left(\sum_{k=0}^{n-1} (-1)^{n-k+1}\beta_{k,n}\gamma^{k}e^{2k}\right)
		\end{equation}
		\noindent with $ e=e(\widetilde{\mo{}}^{+}(1) ) $, $\tilde e=e(\widetilde{\mo{}}^{+}(2) )$  and $ \beta_{k,n} $ defined by the recurrence relation:
		\[ \begin{cases*}
			\beta_{0,n}=n & $ \forall \ n\geq 0 $ \\
			\beta_{k,n}=\sum_{j=1}^{n} j \cdot \beta_{k-1, n-j} & $\forall \ n \geq  k>0 $\\
			\beta_{k,n}=0 & else 
		\end{cases*} \]
		\item We also have:
		\begin{equation}\label{ch3:_even_odd_euler_classes_relation}
			\tilde{e}^2=-4e+\gamma e^4
		\end{equation}
		
	\end{enumerate}
	
\end{pr}
\begin{proof}
	Recall that we can compute the Witt theory of $ {B}N $ by its finite level approximations $ {B}_mN $. In particular, we get that the Cartan Sum Formula and Ananyevskiy's Lemma hold true for $ {B}N $ too.\\
	The first assertion of the proposition follows from the fact that $ \rho^+_m $ and $ \rho_m^{-} $ are isomorphic as representation through $ (x,y)\mapsto (-x,y) $, but this map induces a $ (-1) $ on the determinant, so we get $ e(\widetilde{\mo{}}^{-}(m))=-e(\widetilde{\mo{}}^{+}(m)) $. \\
	The decomposition \cref{ch3:_decomposition_Om_for_recursive formulas} and the Cartan Sum Formula   \ref{ch3:_Cartan_Sum_Formula_Twisted}, for $ m>1 $, gives us:
	\[ b_1\left( \widetilde{\mo{}}^{+}(m) \otimes \widetilde{\mo{}}^{+}(1)^{\otimes 2} \right)=b_1(\widetilde{\mo{}}^{+}(m+2))+2b_1(\widetilde{\mo{}}^{+}(m))+ b_1(\widetilde{\mo{}}^{+}(m)) \]
	\noindent while for $ m=1 $ we have:
	\[  b_1\left( \widetilde{\mo{}}^{+}(1) \otimes \widetilde{\mo{}}^{+}(1)^{\otimes 2} \right)=b_1(\widetilde{\mo{}}^{+}(3))+3b_1(\widetilde{\mo{}}^{+}(1)) \]
	At the same time, using Ananyevskiy's \cref{ch3:_Ana_Stab_Op_Lemma} we have:
	\[ 		b_1(\widetilde{\mo{}}^{+}(m)\otimes \widetilde{\mo{}}^{+}(1)^{\otimes 2})= \gamma b_1(\widetilde{\mo{}}^{+}(m))b^2  \]
	Putting this all together we get:
	\[ b_1(\widetilde{\mo{}}^{+}(3))=-3e+\gamma e^{3} \]
	\noindent and for $ m>1 $:
	\begin{equation}\label{ch3:_recursion_by_Ana_Cartan}
		b_1(\widetilde{\mo{}}^{+}(m+2))=(\gamma e^2-2)b_1(\widetilde{\mo{}}^{+}(m))-b_1(\widetilde{\mo{}}^{+}(m-2)) 
	\end{equation}
	
	\noindent From this by induction is not difficult to see that $ e $ and $ \tilde e $ divides $ b_1(\widetilde{\mo{}}^{+}(m+2)) $ for odd or even  $ m>1 $ respectively, but it will be even more clear from the recursive formulas we are going to show.\\
	Now let us consider $ m=2n+1 $. We want to prove the recursive formula \cref{ch3:_Odd_recursive_formula}. We just saw the formula for $ m=3 $. Now we proceed by induction on $ n $, let us suppose we know the formula for $ m=2n+1 $ and we want to prove it for $ m+2=2(n+1)+1=2n+3 $. Using the induction hypothesis on \cref{ch3:_recursion_by_Ana_Cartan}, we have:
	\begin{align*}
		&b_1(\widetilde{\mo{}}^{+}(m+2))=(\gamma b^2-2)b_1(\widetilde{\mo{}}^{+}(m)) -b_1(\widetilde{\mo{}}^{+}(m-2))=\\
		&=  \sum_{k=0}^{n} (-1)^{n-k}\alpha_{k,n}\gamma^{k+1}b^{2k+3} -2\sum_{k=0}^{n} (-1)^{n-k}\alpha_{k,n}\gamma^{k}b^{2k+1}  -\sum_{k=0}^{n-1} (-1)^{n-1-k}\alpha_{k,n-1} \gamma^{k}b^{2k+1}=\\
		&=(-1)^{n+1}(-\alpha_{0,n-1}+2\alpha_{0,n})b + \sum_{k=1}^{n}(-1)^{n+1-k} (-\alpha_{k,n-1}+2\alpha_{k,n})\gamma^{k}b^{2k+1}+ \\
		&\ \ \ +\sum_{h=1}^{n+1}(-1)^{n+1-h}\alpha_{h-1,n}\gamma^{h}b^{2h+1} + \alpha_{n,n}\gamma^{n+1}b^{2n+3}=\\
		&= (-1)^{n+1}(2n+3)b +\sum_{k=1}^{n}(-1)^{n+1-k}\left( -\alpha_{k,n-1}+2\alpha_{k,n}+\alpha_{k-1, n} \right)\gamma^{k}b^{2k+1}+ \alpha_{n,n}\gamma^{n+1}b^{2n+3}
	\end{align*}
	\noindent and it is easy to see that the coefficients $ \alpha_{0,n+1} $ and $ \alpha_{n+1,n+1} $ are already of the desired form. The only thing left to check is that:
	\[  -\alpha_{k,n-1}+2\alpha_{k,n}+\alpha_{k-1, n} \]
	\noindent is actually:
	\[ \sum_{j=1}^{n+1} j \cdot \alpha_{k-1, n+1-j}  \]
	\noindent for $ k=1,\ldots, n $. using the induction hypothesis:
	\begin{align*}
		&-\alpha_{k,n-1}+2\alpha_{k,n}+\alpha_{k-1, n} =\\
		&=-\sum_{j=1}^{n-1} j \cdot \alpha_{k-1, n-1-j} + 2 \sum_{j=1}^{n} j\cdot \alpha_{k-1, n-j}+\alpha_{k-1, n}=\\
		&= -\sum_{j=1}^{n-1} j \cdot \alpha_{k-1, n-1-j} + 2 \sum_{h=1}^{n-1}(h+1)\cdot \alpha_{k-1, n-1-h} + 2\alpha_{k-1, n-1}+\alpha_{k-1, n}=\\
		&=\sum_{j=1}^{n-1} (-j+2j+2)\cdot \alpha_{k-1, n-1-j}+2\cdot \alpha_{k-1, n-1}+\alpha_{k-1, n}=\\
		&=\sum_{j=0}^{n-1} (j+2)\cdot \alpha_{k-1, n-1-j}+2\cdot \alpha_{k-1,n+1-2}+\alpha_{k-1,n+1-1}=
		\sum_{r=1}^{n+1} r \cdot \alpha_{k-1, n+1-r}
	\end{align*}
	\noindent And this completes the proof of the claim for the odd case $ m=2n+3 $.\\
	
	For the even case $ m=2n $ the proof is basically the same and we will leave it as an exercise.\\
	
	Lastly, again using Ananyevskiy's \cref{ch3:_Ana_Stab_Op_Lemma} and the Cartan sum formula on the decomposition from \cref{ch3:_Decomposition_O(2)_squared}, we have:
	\[ \gamma \tilde{e}^2e=b_1\left( \widetilde{\mo{}}^{+}(2)^{\otimes 2} \otimes \widetilde{\mo{}}^{+}(1)\right)= b_1(\widetilde{\mo{}}^{+}(5))+b_1\left( \widetilde{\mo{}}^{+}(3) \right)+b_1\left( \widetilde{\mo{}}^{-}(1) \right) \]
	\noindent And from what we proved before, this means:
	\[ \gamma \tilde{e}^2 e=(5e-5\gamma e^3+\gamma e^5)+(-3e+\gamma e^3)-2e=-4\gamma e^3+\gamma e^5\]
	\noindent but since $ \gamma $ is an isomorphism, multiplication by $ \gamma e $ is injective on the first summand of $ \mr{KW}^{\bullet}(BN)\simeq \mr{KW}^{\bullet}(S)\llbracket e \rrbracket \oplus q_0 \cdot \mr{KW}^{\bullet}(S) $ and hence:
	\[ \tilde{e}^2=-4 e^2+ \gamma e^4 \]
	
\end{proof}
\begin{rmk}
	Here is a table of easy examples of coefficients $ \alpha_{k,n} $ of the odd Euler classes:
	\[ \begin{array}{r|ccccccc}
		\hbox{\mydiagbox{$n$}{$k$}} & 0& 1 & 2 & 3 & 4 & 5 & 6\\ \hline
		0 & 1 &  &&&&&\\ 
		1 & 3 & 1 & & &&& \\ 
		2 & 5  & 5 & 1 & &&& \\ 
		3 & 7 & 14 & 7 & 1 && &\\
		4 & 9 & 30 & 27 & 9 & 1 && \\
		5& 11 & 55 & 77 & 44 & 11 & 1 &\\
		6 & 13 & 91 & 182&  156 & 65 & 13 & 1 
	\end{array} \]
	It is not hard to show that $ \alpha_{1,n}=\sum_{j=1}^{n} j^2=\frac{n(n+1)(2n+1)}{6}, \alpha_{n-2,n}=n+1(2n+1), \alpha_{n-1,n}=2n+1, \alpha_{n,n}=1 $. Recall that the coefficient of the lowest degree term in the recursive formula will be $ (-1)^{n}(2n+1)=(-1)^n\cdot m $ for $ m=2n+1 $, so the lowest degree term of $ e^{\mr{KW}}\left( \widetilde{\mo{}}^{+}(m) \right) $ will have the same form of the Witt cohomology Euler class computed in \cite[Theorem 7.1]{Motivic_Euler_Char}.\\
	For even Euler classes, the table of the first coefficients looks like:
	
	\[ \begin{array}{r|ccccccc}
		\hbox{\mydiagbox{$n$}{$k$}} & 0& 1 & 2 & 3 & 4 & 5 & 6\\ \hline
		1 & 1 &  &&&&&\\ 
		2 & 2 & 1 & & &&& \\ 
		3 & 3  & 4 & 1 & &&& \\ 
		4 & 4 & 10 & 6 & 1 && &\\
		5 & 5 & 20 & 21 & 8 & 1 && \\
		6& 6 & 35 & 56 & 36 & 10 & 1 &\\
		7 & 7 & 56 & 126&  120 & 55 & 12 & 1 
	\end{array} \]
	Again the coefficient of the lowest degree term in the recursive formula will be $ (-1)^{n-1}n=(-1)^{n-1}\frac{m}{2} $ for $ m=2n $, so the lowest degree term of $ e^{\mr{KW}}\left( \widetilde{\mo{}}^{+}(m) \right) $ will have the same form as the formula obtained for the Witt cohomology Euler class in  \cite[Theorem 7.1]{Motivic_Euler_Char}. But this is not just a coincidence:
	\begin{co}[{\cite[Theorem 7.1]{Motivic_Euler_Char}}]
		Let $ H\mc W $ be the spectrum representing Witt-sheaf cohomology. Denote by $ e^{H\mc W}:=e^{H\mc W}\left(\widetilde{\mo{}}^{+}(1)\right) $ and $ \tilde{e}^{H\mc W}:=e^{H\mc W}\left( \widetilde{\mo{}}^{+}(2) \right) $. Then:
		\[ e^{H\mc W}\left( \widetilde{\mo{}}^{+}(m) \right)=\begin{cases*}
			(-1)^{\frac{m-1}{2}}\cdot m\cdot e^{H\mc W} & for $ m $ odd\\
			(-1)^{\frac{m+2}{2}} \cdot \frac{m}{2} \cdot \tilde{e}^{H\mc W} & for $ m $ even
		\end{cases*} \]
	\end{co}
	\begin{proof}
		By \cite[Remark 6.7]{Fasel_Haution}, we know that the formal ternary law of Hermitian K-theory recovers the computations of \cite[Lemma 8.2]{Ananyevskiy_FTL_Witt}.  Setting $ \gamma=0 $, we recover the formal ternary law for Witt sheaf cohomology by \cite[Theorem 3.3.2]{Déglise_Fasel_Borel_Char}. Using the formulas for Witt sheaf cohomology  obtained from \cref{ch3:_Ana_Stab_Op_Lemma} for $ \gamma=0 $, we can proceed by induction as we did for $ \mr{KW} $ in \cref{ch3:_recursive_formulas_proposition}. In this way we recover the (closed form) formulas of \cite[Theorem 7.1]{Motivic_Euler_Char} as claimed.
	\end{proof}

	
\end{rmk}

\begin{rmk}\label{ch3:_Notation_e_and_Gamma_N}
	Clearly $ \widetilde{\mo{}}^+(1)\simeq p^*E_2 $ where $ p: {B}N \rightarrow {B}SL_2 $ and $ E_2 $ is the bundle arising from the universal bundle of $ \mc BSL_2 $, so the class $ e:=e\left( \widetilde{\mo{}}^+(1) \right) $  is the free generator of $ \mr{KW}^{\bullet}\left( {B}N \right) $ by \cite[Corollary 3.26]{Witt_Loc_PhD}. Similarly the class $ \tilde e:=e\left( \widetilde{\mo{}}^+(2) \right) $ is the free generator of $ \mr{KW}^{\bullet}\left( {B}N; \gamma_N \right) $ again by \cite[Corollary 3.26]{Witt_Loc_PhD} and \cite[Lemma 6.1]{Motivic_Euler_Char}.\\
	The reason we are so interested in these rank 2 vector bundles $ \widetilde{\mo{}}^{\pm}(m)  $ is that the $ \rho_m^{\pm} $'s classify irreducible representations of $ BN $ with a choice of the determinant. This together with \cite[\textit{Splitting Principle, \S 9 }]{Ananyevskiy_PhD_Thesis}, \cref{ch3:_4.1_MEC}, and the structure theorem \cite[Corollary 3.26]{Witt_Loc_PhD}, tells us that the characteristic classes of these bundles, and their symmetric powers, will recover all the characteristic classes of any vector bundle on $ BN $.
\end{rmk}

\begin{rmk}\label{ch3:_Euler_Classes_Sym_}
	Let us notice that given any rank two vector bundle $ E\rightarrow X $ over some $ X $, we have $ e(\sym^{k}E)=0 \in \mr{KW}^{k+1}(X; \det^{-1}\sym^kE) $ for $ k $ even since $ \sym^kE $ will have rank $ k+1 $ and thus we can apply \cite[Lemma 4.3]{Motivic_Euler_Char}. For $ k=2r+1 $ odd, $ e(\sym^kE) $, seen as a polynomial in $ e:=e(\widetilde{\mo{}}^{+}(1)) $, will have as lowest degree term $ k!! \cdot e^{r+1} $, where $ k!!:=k\cdot (k-2)\cdot\ldots \cdot 3 \cdot 1 $. The proof of this fact is basically the same proof used in \cite[Theorem 8.1]{Motivic_Euler_Char} using the concrete description we have for the lowest degree term  of $ e(\widetilde{\mo{}}^{+}(m)) $ given by \cref{ch3:_recursive_formulas_proposition} and the fact we remarked above that they recover the same computations done in \cite[Theorem 7.1]{Motivic_Euler_Char}).\\
\end{rmk}

	\cleardoublepage
	\addcontentsline{toc}{section}{References}

	\printbibliography	\thispagestyle{empty}

@article{Motivic_Euler_Char, 
	title={Motivic Euler Characteristics and Witt-Valued Characteristic Classes}, 
	volume={236}, 
	DOI={10.1017/nmj.2019.6}, 
	journal={Nagoya Mathematical Journal}, 
	publisher={Cambridge University Press}, 
	author={Levine, Marc}, 
	year={2019}, 
	pages={251–310}}

@article{Morel-Voevodsky,
	Author = {Morel, Fabien and Voevodsky, Vladimir},
	Journal = {Publications Math{\'e}matiques de l'Institut des Hautes {\'E}tudes Scientifiques},
	Number = {1},
	Pages = {45--143},
	Title = {A1-homotopy theory of schemes},
	Volume = {90},
	Year = {1999}}

@article{Panin-Walter,
	author = {I. Panin and C. Walter },
	title = {On the Commutative Ring Spectrum {$ \mr{BO} $}},
	journaltitle = {Algebra i Analiz},
	date = {2018},
	number = {6},
	issue = {30},
	doi = {https://doi.org/10.1090/spmj/1578},
}

@thesis{Kumar_PhD,
	author = {A. Kumar},
	title = {On the Motivic Spectrum BO and Hermitian K-Theory},
	type = {PhD Thesis},
	institution = {Universit{\"a}t Osnabr{\"u}ck},
	date = {2020},
}

@article{Ananyevskiy_Witt_MSp_Reations,
	author = {Ananyevskiy, Alexey},
	title = {On the relation of special linear algebraic cobordism to Witt groups},
	journaltitle = {Homology, Homotopy and Applications},
	volume={18},
	issue={1},
	pages={205-230},
	publisher={International Press of Boston},
	date = {2016}
}

@misc{Asok-Deglise-Nagel,
	doi = {10.48550/arxiv.1812.09574},
	
	author = {Asok, Aravind and D{\'e}glise, Fr{\'e}d{\'e}ric and Nagel, Jan},
	
	title = {The homotopy Leray spectral sequence},
	
	publisher = {arXiv},
	
	year = {2018},
	
	copyright = {arXiv.org perpetual, non-exclusive license}
}

@article{Ananyevskiy_PhD_Thesis, 
	title={The special linear version of the projective bundle theorem}, 
	volume={151}, 
	DOI={10.1112/S0010437X14007702}, 
	number={3}, journal={Compositio Mathematica}, 
	publisher={London Mathematical Society}, 
	author={Ananyevskiy, Alexey}, 
	year={2015}, 
	pages={461–501}}

@Article{DJK,
	Author = {D{\'e}glise, Fr{\'e}d{\'e}ric and Jin, Fangzhou and Khan, Adeel A.},
	Title = {Fundamental classes in motivic homotopy theory},
	FJournal = {Journal of the European Mathematical Society (JEMS)},
	Journal = {J. Eur. Math. Soc. (JEMS)},
	ISSN = {1435-9855},
	Volume = {23},
	Number = {12},
	Pages = {3935--3993},
	Year = {2021},
	Language = {English},
	DOI = {10.4171/JEMS/1094},
	zbMATH = {7445598},
	Zbl = {1483.14040}
}

@misc{Ananyevskiy_SL_oriented,
	doi = {10.48550/arxiv.1901.01597},
	
	author = {Ananyevskiy, Alexey},
	
	title = {SL-oriented cohomology theories},
	
	publisher = {arXiv},
	
	year = {2019},
	
	copyright = {arXiv.org perpetual, non-exclusive license}
}

@misc{VLF_Levine,
	doi = {10.48550/arxiv.2203.15887},
	
	
	author = {Levine, Marc},
	
	
	title = {Virtual Localization in equivariant Witt cohomology},
	
	publisher = {arXiv},
	
	year = {2022},
	
	copyright = {Creative Commons Attribution Non Commercial Share Alike 4.0 International}
}

@misc{levine2017intrinsic,
	doi = {10.48550/arxiv.1703.03056},
	author = {Levine, Marc},
	title = {The intrinsic stable normal cone},
	publisher = {arXiv},
	year = {2017}
}

@misc{Levine_Atiyah-Bott,
	doi = {10.48550/arxiv.2203.13882},
	author = {Levine, Marc},
	title = {Atiyah-Bott localization in equivariant Witt cohomology},
	publisher = {arXiv},
	year = {2022}
}

@article{Ananyevskiy_FTL_Witt,
	author = {Alexey Ananyevskiy},
	title = {{Stable operations and cooperations in derived Witt theory with rational coefficients}},
	volume = {2},
	journal = {Annals of K-Theory},
	number = {4},
	publisher = {MSP},
	pages = {517 -- 560},
	year = {2017},
	doi = {10.2140/akt.2017.2.517}
}

@article{Déglise_Fasel_Borel_Char, 
	title={The Borel Character}, 
	DOI={10.1017/S1474748021000281}, 
	journal={Journal of the Institute of Mathematics of Jussieu}, 
	publisher={Cambridge University Press}, 
	author={D{\'e}glise, Fr{\'e}d{\`{e}}ric and Fasel, Jean}, 
	year={2021}, 
	pages={1–51}}

@article{Ananyevskiy_Pushforwards_Eta_Inverted,
	Author = {Ananyevskiy, Alexey},
	Journal = {Manuscripta Mathematica},
	Number = {1},
	Pages = {21--44},
	Title = {On the push-forwards for motivic cohomology theories with invertible stable Hopf element},
	Volume = {150},
	Year = {2016}}

@article{schlichting_Hermitian_2010, 
	title={Hermitian K-theory of exact categories}, 
	volume={5}, DOI={10.1017/is009010017jkt075}, 
	number={1}, journal={Journal of K-Theory}, 
	publisher={Cambridge University Press}, 
	author={Schlichting, Marco}, 
	year={2010}, 
	pages={105–165}}

@PhdThesis{Viergever_PhD,
	author = 	{Viergever, Anna M.},
	title = 	{The quadratic Euler characteristic of a smooth projective same-degree complete intersection and  motivic Donaldson-Thomas invariants of $ \mathbb P^3 $},
	year = 	{2023},
	month = 	{Jul},
	language = 	{en}
}

@article{Levine_Raksit_Gauss_Bonnet,
	author = {Marc Levine and Arpon Raksit},
	title = {{Motivic Gauss–Bonnet formulas}},
	volume = {14},
	journal = {Algebra {\&} Number Theory},
	number = {7},
	publisher = {MSP},
	pages = {1801 -- 1851},
	year = {2020},
	doi = {10.2140/ant.2020.14.1801},
	URL = {https://doi.org/10.2140/ant.2020.14.1801}
}

@article{Panin-Walter_MSL_MSp,
	author = { I. Panin and C. Walter },
	title = {On the algebraic cobordism spectra 
	MSL
	and MSp},
	journaltitle = {Algebra i Analiz, tom 34 (2022), nomer 1},
	date = {2022}
}

@book{Bachmann_Hoyois_Norms_MHT,
	title = {Norms in motivic homotopy theory},
	author = {Tom Bachmann and Marc Hoyois},
	volume = {425},
	year = {2021},
	address = {Paris},
	publisher = {Soci{\'e}t{\'e} Math{\'e}matique de France},
}

@ARTICLE{Bachmann_Wickelgren,
	title     = "Euler classes: Six-functors formalism, dualities, integrality
	and linear subspaces of complete intersections",
	author    = "Bachmann, Tom and Wickelgren, Kirsten",
	journal   = "J. Inst. Math. Jussieu",
	publisher = "Cambridge University Press (CUP)",
	pages     = "1--66",
	month     =  {jul},
	year      =  {2021},
	language  = "en"
}

@misc{Fasel_Haution,
	title={The stable Adams operations on Hermitian K-theory}, 
	author={Jean Fasel and Olivier Haution},
	year={2023},
	eprint={2005.08871},
	archivePrefix={arXiv},
}

@article{Haution_Odd_VB, 
	title={Odd rank vector bundles in eta-periodic motivic homotopy theory}, 
	DOI={https://doi.org/10.48550/arXiv.2203.06021}, 
	journal={Journal of the Institute of Mathematics of Jussieu}, 
	author={Haution, Oliver}, 
	year={2023}, 
	pages={To appear}
}

@article {Brion_lin_pic,
	AUTHOR = {Brion, Michel},
	TITLE = {On linearization of line bundles},
	JOURNAL = {J. Math. Sci. Univ. Tokyo},
	FJOURNAL = {The University of Tokyo. Journal of Mathematical Sciences},
	VOLUME = {22},
	YEAR = {2015},
	NUMBER = {1},
	PAGES = {113--147},
	ISSN = {1340-5705},
}

@ARTICLE{Ananyevskiy_Thom_iso,
	title     = "Thom isomorphisms in triangulated motivic categories",
	author    = "Ananyevskiy, Alexey",
	journal   = "Algebr. Geom. Topol.",
	publisher = "Mathematical Sciences Publishers",
	volume    =  21,
	number    =  4,
	pages     = "2085--2106",
	month     =  aug,
	year      =  2021,
	language  = "en"
}

@misc{ChoDA24,
	title={Non-representable six-functor formalisms}, 
	author={Chirantan Chowdhury and Alessandro D'Angelo},
	year={2024},
	eprint={2409.20382},
	archivePrefix={arXiv},
	primaryClass={math.AG},
	url={https://arxiv.org/abs/2409.20382}, 
	shorthand={CD{'}\!A24}
}

@article{Chowdhury24,
	author = {Chowdhury, Chirantan},
	year = {2024},
	month = {05},
	pages = {1-22},
	title = {Motivic homotopy theory of algebraic stacks},
	volume = {9},
	journal = {Annals of K-Theory},
	doi = {10.2140/akt.2024.9.1}
}

@unpublished{Motivic_Vistoli,
	title={A Motivic Vistoli's Lemma: after Levine}, 
	author={Alessandro D'Angelo},
	date={2024},
	url={https://sites.google.com/view/alessandro-dangelo/home/research}
}

@misc{Witt_Loc_PhD,
	title={Virtual Localisation Formula for $SL_{\eta}$-oriented Spectra}, 
	author={Alessandro D'Angelo},
	year={2024},
	eprint={},
	archivePrefix={arXiv},
	primaryClass={math.AG},
	url={https://sites.google.com/view/alessandro-dangelo/home/research}
}
	
\end{document}